\documentclass{extarticle}
\usepackage{amsfonts,amssymb,amsmath,amsthm}
\usepackage[margin=3cm,letterpaper]{geometry}
\usepackage[dvipsnames]{xcolor}
\usepackage[colorlinks=true,linkcolor=RoyalBlue,citecolor=PineGreen,urlcolor=blue]{hyperref}
\usepackage[numbers,sort]{natbib}

\usepackage{enumitem}
\usepackage{comment}

\usepackage{bm}

\newcommand{\NN}{\mathbb{N}}

\newtheorem{question}{Question}
\newtheorem*{answer}{Answer}

\newtheorem{proposition}{Proposition}[section]
\newtheorem{lemma}[proposition]{Lemma}
\newtheorem{theorem}[proposition]{Theorem}
\newtheorem{corollary}[proposition]{Corollary}

\theoremstyle{definition}
\newtheorem{definition}[proposition]{Definition}
\newtheorem{example}[proposition]{Example}
\newtheorem{remark}[proposition]{Remark}

\title{Solving difference equations in sequences:\\ Universality and Undecidability}

\date{}

\author{Gleb Pogudin\footnote{\href{mailto:pogudin.gleb@gmail.com}{pogudin.gleb@gmail.com}, Department of Computer Science, National Research University Higher School of Economics, Moscow, Russia; current address: LIX, CNRS, Ecole Polytechnique, Institut Polytechnique de Paris, France}, Thomas Scanlon\footnote{\href{mailto:scanlon@math.berkeley.edu}{scanlon@math.berkeley.edu}, University of California at Berkeley, Department of Mathematics, Berkeley, USA}, Michael Wibmer\footnote{\href{mailto:wibmer@math.tugraz.at}{wibmer@math.tugraz.at}, Institut of Analysis and Number Theory, Graz University of Technology, Graz, Austria}}

\begin{document}

\maketitle

\begin{abstract}
	\let\thefootnote\relax\footnotetext{{\em \hspace{-3mm} Mathematics Subject Classification Codes:} 
	12H10, 
	39A10, 
	13P25, 
	14Q20, 
	68Q40, 
	03D35. 
	
			{\em Key words and phrases}: Algebraic difference equations, solutions in sequences, undecidability, difference Nullstellensatz, radical difference ideal membership problem.}
    We study solutions of difference equations in the rings of sequences and, more generally, solutions of equations with a monoid action in the ring of sequences indexed by the monoid.
    This framework includes, for example, difference equations on grids (e.g., standard difference schemes) and difference equations in functions on words.
    
    On the universality side, we prove a version of strong Nullstellensatz for such difference equations under the assuption that the cardinality of the ground field is greater than the cardinality of the monoid and construct an example showing that this assumption cannot be omitted.
    
    On the undecidability side, we show that the following problems are undecidable:
    \begin{itemize}
      \item testing radical difference ideal membership or, equivalently, determining whether a given difference polynomial vanishes on the solution set of a given system of difference polynomials;
      \item determining consistency of a system of difference equations in the ring of real-valued sequences;
      \item determining consistency of a system of equations with action of $\mathbb{Z}^2$, $\mathbb{N}^2$, or the free monoid with two generators in the corresponding ring of sequences over any field of characteristic zero.
    \end{itemize}
\end{abstract}

\section{Introduction}


An ordinary difference ring $(A,\sigma)$ 
is a commutative ring $A$ equipped with a 
distinguished ring endomorphism 
$\sigma:A \to A$.  The most basic example of a 
difference ring is the 
ring  $\mathbb{C}^\mathbb{N}$ of sequences
of complex numbers with 
$\sigma$ defined by $(a_i)_{i \in \mathbb{N}} 
\mapsto (a_{i+1})_{i \in \mathbb{N}}$.  More 
generally, if $\phi:X \to X$ is any 
self-map on a set $X$ and $A$ is the ring of 
complex valued functions on $X$, then $\sigma:A \to A$
defined by $f \mapsto f \circ \phi$ is a 
difference ring.   The special case where $X = \mathbb{R}$ is
the real line and $\phi$ is given by $\phi(x) = x  + 1$ 
gives the operator defined by $f(t) \mapsto f(t+1)$
 and explains the origin of the name 
``difference ring'' in that the discrete difference
operator $\Delta$ defined by 
$f(t) \mapsto f(t+1) - f(t)$ may be expressed as 
$\Delta = \sigma - \operatorname{id}$.    Generalizing 
to allow for additional operators, we might consider
partial difference rings $(A,\sigma_1,\ldots,\sigma_n)$
with several distinguished ring endomorphisms
$\sigma_j:A \to A$.  Natural instances of such 
partial difference rings with commuting operators
include rings of sequences indexed by $n$-tuples
of natural numbers and the rings of $n$-variable 
functions.  There are also natural examples of 
such partial difference rings with non-commuting 
difference operators coming from number theory, 
the theory of iterated function systems, and 
symbolic dynamics.  

We may think of a partial 
difference ring $(A,\sigma_1,\ldots,\sigma_n)$
as the ring $A$ given together with an 
action by ring endomorphisms of $M_n$, the 
free monoid on $n$ generators.   If we 
require that these operators commute, then 
this may be seen as an action by $\NN^n$.  
Likewise, if we require that the operators 
are, in fact, ring automorphisms, then it is an 
action by $F_n$, the free group on $n$-generators.

As with algebraic 
and differential equations, the most basic 
problems for difference equations come down 
to solving these equations in some specified
difference ring. As a preliminary, difficult
subproblem, one must determine whether the equations
under consideration admit any solutions at all.   
In the optimal cases, solvability of a system of 
equations is equivalent to a suitable Nullstellensatz
in some associated ring of polynomials 
(respectively, differential polynomials or
difference polynomials).   While in the case of 
polynomial equations in finitely many variables these 
problems admit well known solutions, for difference 
and differential equations and their relatives, there
are subtle distinctions between those problems which may 
be solved and those for which no algorithm exists.   

In many cases, the problems we are considering may be 
resolved by analyzing the associated first-order theories.  
The prototypical decidability theorems for equations
are Tarski's theorems on the decidability and 
completeness of the theories of real closed fields 
and of algebraically closed fields of a 
fixed characteristic~\cite{Tarski}.     This logical theorem is 
complemented algebraically by Hilbert's Nullstellensatz
which gives a precise sense in which implications for
systems of polynomial equations may be expressed in 
terms of ideal membership problems.

Theorems analogous to Tarski's
are known for difference and differential
\emph{fields}.   The theories of difference fields,
of differential fields of characteristic zero, and 
even of partial differential fields of characteristic
zero and of difference-differential fields of 
characteristic zero are known to have model companions (see~\cite{ChHr99, C15, BM07, Sanchez16}).
Moreover, for each of these theories, quantifier 
simplification theorems (and even full quantifier
elimination theorems in the case of differential 
fields) are known.  From these results one may deduce
on general grounds the existence of algorithms for 
determining the consistency of systems of 
difference (respectively, differential or 
difference-differential) equations in such fields 
and explicit, if not always efficient, such algorithms
may be extracted from the more geometric presentations
of the axioms.  Better algorithms based on 
characteristic set methods are known~\cite{KolchinBook, GLY2009, Gao2009}.

From the algebraic point of view, the consistency
checking problem may be expressed in terms of some 
form of a Nullstellensatz.   For example, the weak 
form of the classical Nullstellensatz of Hilbert 
says that if $K$ is an algebraically closed field and 
and $f_1, \ldots, f_\ell \in K[x_1, \ldots, x_n]$ is a 
sequence of polynomials in the finitely many variables
$x_1, \ldots, x_n$ then the system of equations 
\begin{equation}
\label{algeq}
f_1(\mathbf{x}) = \cdots = f_\ell(\mathbf{x}) = 0
\end{equation}
(where we have written $\mathbf{x} = (x_1, \ldots, x_n)$) 
has a solution in $K$ if and only if $1$ does \emph{not}
belong to the ideal $\langle f_1, \ldots, f_\ell \rangle$ generated 
by $f_1, \ldots, f_\ell$.
The latter condition can be verified by a linear algebra computation (see~\cite{Jelonek} and references therein).

Hilbert's Nullstellensatz takes a stronger form in that 
one may reduce implications between systems of 
equations to explicit computations in polynomial
rings.  That is, given equations as above and 
$g \in K[\mathbf{x}]$ any polynomial, then $g$ vanishes
on every solution to Equation~\eqref{algeq} if and 
only if $g \in \sqrt{\langle f_1, \ldots, f_\ell\rangle}$, the 
radical of the ideal generated by $f_1, \ldots, 
f_\ell$ .   
Similar results are known 
for equations in differential, difference, and 
difference-differential fields.
The situation is murkier if we consider partial difference equations, that is, 
difference equations with respect to several distinguished ring 
endormorphisms.  
It is noted in~\cite{Hr:MM} that the theory of difference fields with 
respect to finitely many distinguished endomorphisms has a 
model companion, and, in fact, a simple variant of the method
for determining the consistency of systems of difference equations
for ordinary difference equations extends to this case of partial difference 
equations.  However, if the distinguished endomorphisms are required to 
commute, then no such model companion exists~\cite{Kikyo}.  

Rings of sequences are among the most natural places to look for solutions
of difference equations.
In particular, algorithms for detecting the solvability of finite
systems of difference equations in sequence rings are 
available~\cite{difference_elimination}.
However, the general problem of solving equations in sequences is much more complicated than the analogous problem for difference fields:
whenever $K$ is infinite, the
first-order theory of the sequence ring $K^\NN$ 
regarded in the language of difference rings is 
undecidable~\cite[Proposition~3.5]{HrPo}.    

The staring point for us was a recent paper~\cite{difference_elimination} that contains the following results about solving difference equations in sequences:
\begin{itemize}
    \item \emph{The weak Nullstellensatz}~\cite[Theorem~7.1]{difference_elimination}: for any algebraically closed difference field $(K,\sigma)$ and
a finite set $S$ of difference equations over $K$, there is a 
solution in $K^\NN$ to the system $S$ if and
only if the difference ideal generated by $S$ is proper;
    \item An effective bound~\cite[Theorem~3.4]{difference_elimination} that yields an \emph{algorithm} for deciding whether a difference ideal given by its generators is proper and, consequently, an \emph{algorithm} for deciding consistency of a finite system of difference equations in $K^{\NN}$.
\end{itemize}
Remarkably, while the proof of the weak difference Nullstellensatz
is rather routine for $K$ uncountable, the result holds 
for arbitrary $K$. 

In this paper, we answer several natural questions aimed at extending the above results about solving difference equations in sequences.
\begin{question}[weak Nullstellensatz $\to$ strong Nullstellensatz]
  If $f_1, \ldots, f_\ell$, and $g$ are difference polynomials over an 
algebraically closed difference field $K$ and $g$ vanishes on every 
solution to the system of difference equations $f_1(\mathbf{x}) = 
\cdots = f_\ell(\mathbf{x}) = 0$ in $K^\mathbb{N}$, must 
$g$ belong to the radical of the difference ideal generated by 
$f_1, \ldots, f_\ell$?   
\end{question}

\begin{answer}
Depends on the cardinality of $K$ (Theorems~\ref{thm:nullstellensatz_over_constants} and~\ref{theo:counter strong_in_main}).
\end{answer}

\noindent
More precisely, we show that the answer is Yes if $K$ is uncountable (Theorem~\ref{thm:nullstellensatz_over_constants}) and give an example that shows that the answer is No for $K = \bar{\mathbb{Q}}$ (Theorem~\ref{theo:counter strong_in_main}).
It is interesting to compare this result with the weak Nullstellensatz~\cite[Theorem~7.1]{difference_elimination} that holds for a ground field of any cardinality but the proof for the countable case is much harder than the proof for the uncountable case.

\begin{question}[testing consistency $\to$ testing radical difference ideal membership]
  Is there an algorithm that, given difference polynomials $f_1, \ldots, f_\ell$, and $g$, decides whether $g$ belongs to the radical difference ideal generated by $f_1, \ldots, f_\ell$?
\end{question}

\begin{answer}
No (Theorem~\ref{thm:stong_undecidable}).
\end{answer}

\noindent
This result contrasts not only with the existence of an algorithm for this problem if $g = 1$ (see \cite[Theorem~3.4]{difference_elimination}) but also with the decidability of the membership problem for radical differential ideals~\cite[p. 110]{Ritt}.
Furthermore, we are aware of only one prior undecidability result for the membership problem in the context of differential/difference algebra~\cite{Umirbaev:AlgorithmicProblemsForDifferentialPolynomialAlgebras}, and this result holds if one considers not necessarily radical ideals and at least two derivations.

\begin{question}[not necessarily algebraically closed $K$]
  Is there an algorithm that, given difference polynomials $f_1, \ldots, f_\ell$ over $\mathbb{R}$, decides whether the system $f_1 = \ldots = f_\ell = 0$ has a solution in $\mathbb{R}^\NN$?
\end{question}

\begin{answer}
No (Theorem~\ref{thm:real}).
\end{answer}

\noindent
Moreover, Theorem~\ref{thm:real} shows that the answer is No if we replace $\mathbb{R}$ with any subfield of $\mathbb{R}$ (including $\mathbb{Q}$).
Again, the situation is different compared to the differential case: the problem of deciding the existence of a real analytic solution of a system of differential equations over $\mathbb{Q}$ is decidable~\cite[\S 4]{S78}.

\begin{question}[index monoids other than $\NN$ or $\mathbb{Z}$]
  Is there an algorithm for deciding consistency of systems of difference equations with respect to actions of $\NN^2$ or the free monoid with two generators when the solutions are sought in the sequences indexed by the corresponding monoid?
\end{question}

\begin{answer}
No (Propositions~\ref{prop:2d} and~\ref{prop:free_monoid}).
\end{answer}

\noindent
Notably, the problem of the solvability of equations in 
the free monoid itself is decidable~\cite{Makanin}.

One of the crucial technical ingredients (used to prove Theorems~\ref{theo:counter strong_in_main} and~\ref{thm:stong_undecidable} and Proposition~\ref{prop:free_monoid}) is Lemma~\ref{lemma: equivalence iteration and strong Nullstellensatz} that connects the membership problem for a radical difference ideal to a problem of Skolem-Mahler-Lech~\cite[\S~2.3]{RecSeq} type for piecewise polynomial maps.
For related undecidability results for dynamical systems associated with other types of maps, see~\cite{M90, KM99, BP18} and references therein.


\section{Preliminaries}

Throughout the paper, $\mathbb{N}$ denotes the set of non-negative integers.

\subsection{Difference rings and equations}

The main objects of the paper are difference equations and their generalizations.
A detailed introduction to difference rings can be found in \cite{Cohn,LevinBook}.
\begin{definition}[Difference rings]
A {\em difference ring} is a pair $(A,\sigma)$ where
$A$ is a commutative ring and $\sigma: A \to A$
is a ring endomorphism.  
 We often abuse notation saying that $A$ is a difference ring when we mean the pair $(A,\sigma)$.   
\end{definition}

The following example of a difference ring will be central in this paper.

\begin{example}[Ring of sequences]\label{ex:sequences}
  If $R$ is any commutative ring, then the sequence rings $R^{\mathbb{N}}$ and $R^{\mathbb{Z}}$ (with componentwise addition and multiplication) are difference rings with 
  $\sigma$ defined by $\sigma ( (x_i)_{i \in \mathbb{N}}) := ( x_{i+1} )_{i \in \mathbb{N}}$ 
  ($\sigma ( (x_i)_{i \in \mathbb{Z}}):= (x_{i+1} )_{i \in \mathbb{Z}}$, respectively).
\end{example}

\begin{definition}[Difference polynomials]\label{def:diff_poly}
Let $A$ be a difference ring.
\begin{itemize}
  \item The free difference $A$-algebra in one generator $X$ over $A$ also called the {\em ring of difference polynomials} in $X$ over $A$, may be realized as the ordinary polynomial ring , $A[ \sigma^j(X) \mid j \in \mathbb{N} ]$, in the indeterminates $\{ \sigma^j(X) \mid j \in \mathbb{N} \}$ with the action $\sigma( \sigma^j(X) ) := \sigma^{j + 1}(X)$. 
  \item Similarly, for $\bm{X} = (X_1, \ldots, X_n)$, one obtains the difference polynomial ring $A[ \sigma^j(\bm{X}) \mid j \in \mathbb{N} ]$ in $n$ variables. 
\end{itemize}
\end{definition}

\begin{definition}
If $(A,\sigma)$ is a difference ring and $F \subseteq A [ \sigma^j(\bm{X}) \mid j \in \mathbb{N} ]$ where $\bm{X} = (X_1, \ldots, X_n)$
is a set of difference polynomials over $A$,
$(A,\sigma) \to (B,\sigma)$ is a map of difference rings, and 
$\mathbf{x} = (x_1, \ldots, x_n) \in B^n$ is an $n$-tuple from $B$, then we 
say that $\mathbf{x}$ is a {\em solution} of the
system $F = 0$ if, under the unique map of difference rings $A [ \sigma^j(\bm{X}) \mid j \in \mathbb{N} ] \to B$ given by extending the given map $A \to B$ and sending $X_i \mapsto x_i$ for $1 \leqslant i \leqslant n$, every element of $F$ is sent to $0$.
\end{definition}

\begin{example}[Fibonacci numbers]
  Consider the Fibonacci sequence $\bm{f} := (1, 1, 2, 3, 5, \ldots) \in \mathbb{C}^{\mathbb{N}}$.
  Then the fact that the sequence satisfies a recurrence $f_{n + 2} = f_{n + 1} + f_n$ can be expressed by saying that $\bm{f}$ is a solution of a difference equation $\sigma^2(X) - \sigma(X) - X = 0$, where $\sigma^2(X) - \sigma(X) - X \in \mathbb{C}[ \sigma^j(X)\mid j \in \mathbb{N} ]$.
\end{example}


\subsection{Rings with a monoid action and equations}

In this paper, we will often be interested in rings of ``sequences'' that would generalize Example~\ref{ex:sequences} to sequences indexed by $\mathbb{Z}^2$ (e.g., difference schemes for PDEs) or any other semigroup.

\begin{definition}[$M$-rings]
Let $M$ be a monoid. 
    A pair $(A, \sigma)$ where $A$ is a commutative ring and $\sigma$ is an action of $M$ on $A$ by endomorphisms is called an \emph{$M$-ring}.
    For every $a \in A$ and $m \in M$, we define the image of $a$ under the endomorphism corresponding to $m$ by $\sigma^m(a)$.
\end{definition}

We note that every difference ring is an $\mathbb{N}$-ring for the monoid $(\mathbb{N}, +)$. A morphism of $M$-rings is a morphism of rings that commutes with the $M$-action.

\begin{example}[Rings of sequences indexed by $\mathbb{N}^2$ and $\mathbb{Z}^2$]\label{ex:seqN2}
  If $R$ is any commutative ring, then the rings $R^{\mathbb{N}^2}$ and $R^{\mathbb{Z}^2}$ are $\mathbb{N}^2$-rings with 
  $\sigma$ defined by 
  \[
    \sigma^{(1, 0)} \bigl( (x_{i, j})_{i, j \in \mathbb{N}} \bigr) := (x_{i + 1, j})_{i, j \in \mathbb{N}} \quad \text{ and } \quad \sigma^{(0, 1)} \bigl( (x_{i, j})_{i, j \in \mathbb{N}} \bigr) := (x_{i, j + 1})_{i, j \in \mathbb{N}}.
  \]
  The action on $R^{\mathbb{Z}^2}$ is defined analogously.
\end{example}

\begin{example}\label{ex:seq_general}
In general, if $R$ is a commutative ring and $M$ a monoid, then the ring $R^M$ of \emph{$M$-sequences} is the commutative ring of all maps from $M$ to $R$ (with componentwise addition and multiplication) and action given by 
$$\sigma^m((x_\ell)_{\ell \in M})=(x_{\ell m})_{\ell\in M}$$
for $m\in M$.
\end{example}

The following example is a special case of Example~\ref{ex:seq_general}.

\begin{example}[Functions on words]
  Let $\Sigma$ be a finite alphabet.
  By $(\Sigma^{\ast}, \cdot)$ we denote the monoid of all words in $\Sigma$ with the operation of concatenation.
  Let $R$ be a commutative ring. 
  Consider the ring of functions $R^{\Sigma^\ast}$ from $\Sigma^\ast$ to $R$ that we will identify with the ring of $\Sigma^\ast$-indexed sequences. 
  Then $R^{\Sigma^\ast}$ can be endowed with a structure of $\Sigma^\ast$ ring as follows
  \[
    \sigma^w\bigl( (x_u)_{u \in \Sigma^\ast} \bigr) := (x_{uw})_{u \in \Sigma^\ast} \text{ for every } w \in \Sigma^\ast.
  \]
\end{example}

\begin{definition}[$M$-polynomials]\label{def:M_poly}
We fix a monoid $M$. Let $A$ be an $M$-ring.
\begin{itemize}
  \item The free $M$-algebra over $A$ in one generator $X$ over $A$ also called the {\em ring of $M$-polynomials} in $X$ over $A$, may be realized as the ordinary polynomial ring , $A[ \sigma^m(X) \mid m \in M ]$, in the indeterminates $\{ \sigma^m(X) \mid m \in M \}$ with the action $\sigma^{m_1}( \sigma^{m_2}(X) ) := \sigma^{m_1m_2}(X)$ for every $m_1, m_2 \in M$. 
  \item Similarly, for $\bm{X} = (X_1, \ldots, X_n)$, one obtains the ring of $M$-polynomials $A[ \sigma^m(\bm{X}) \mid m \in M ]$ in $n$ variables. 
\end{itemize}
\end{definition}

\begin{definition}
We fix a monoid $M$.
If $(A, \sigma)$ is an $M$-ring and $F \subseteq A [ \sigma^m(\bm{X}) \mid m\in M]$ where $\bm{X} = (X_1, \ldots, X_n)$
is a set of $M$-polynomials over $A$,
$(A,\sigma) \to (B,\sigma)$ is a map of $M$-rings, and 
$\mathbf{x} = (x_1, \ldots, x_n) \in B^n$ is an $n$-tuple from $B$, then we 
say that $\mathbf{x}$ is a {\em solution} of the
system $F = 0$ if, under the unique map of $M$-rings $A [\sigma^m(\bm{X}) \mid m \in M] \to B$ given by extending the given map $A \to B$ and sending $X_i \mapsto x_i$ for $1 \leqslant i \leqslant n$, every element of $F$ is sent to $0$. For $f\in A [ \sigma^m(\bm{X}) \mid m\in M]$ we denote the image of $f$ under the above map by $f(\mathbf{x})$.
\end{definition}

\begin{example}[Discrete harmonic functions]
Consider a $\mathbb{C}$-valued function $\bm{x} = (x_{i, j})_{i, j \in \mathbb{Z}^2}$ on the integer lattice. 
It is called a discrete harmonic function~\cite{H49} if, for every $i, j \in \mathbb{Z}^2$, $4x_{i, j} = x_{i + 1, j} + x_{i - 1, j} + x_{i, j + 1} + x_{i, j - 1}$.
The fact that it is a discrete harmonic function can be expressed by the fact that it is a solution of the following $\mathbb{Z}^2$-polynomial
  \[
     4X - \sigma^{(1, 0)}(X) - \sigma^{(-1, 0)}(X) - \sigma^{(0, 1)}(X) - \sigma^{(0, -1)}(X) \in \mathbb{C}[\sigma^m(X) \mid m\in\mathbb{Z}^2 ].
  \]
\end{example}

\begin{example}
Let $M = \{a, b\}^\ast$ be a monoid of binary words with respect to concatenation.
Then the fact that a function $d \colon M \to \mathbb{R}$ is a martingale~\cite[p. 2]{Schnorr} can be expressed by the fact that $d$ is a solution of the following $M$-polynomial
\[
  X - \frac{1}{2}\sigma^a(X) - \frac{1}{2}\sigma^b(X) \in \mathbb{C}[\sigma^m(X) \mid n \in M].
\]
\end{example}


\section{Main results}

\subsection{Universality of sequence rings}\label{subsec:universality}

Let $M$ be a monoid, let $k$ be a field, and let $\bm{X} = (X_1, \ldots, X_n)$. 
For a subset $F$ of $k[\sigma^m(\bm{X}) \mid m \in M]$, we let
\[
  \mathcal{V}(F) = \{ \mathbf{x}\in (k^M)^n \mid f(\mathbf{x})=0\ \forall \ f\in F\}
\]
denote the set of solutions of $F$ in $k^M$ and
for a subset $S$ of $(k^M)^n$, we let
\[
\mathcal{I}(S) = \{f\in k[\sigma^m(\bm{X}) \mid m \in M] \mid f(\bm{x})=0\ \forall \ \bm{x}\in S \}
\]
denote the set of all $M$-polynomials vanishing on $S$.

\begin{theorem}[Strong Nullstellensatz]\label{thm:nullstellensatz_over_constants}
	Let $M$ be a monoid, let $k$ be an algebraically closed field such that $|k| > |M|$, and let $\bm{X} = (X_1, \ldots, X_n)$. 
	Then, for every subset $F$ of $k[\sigma^m(\bm{X}) \mid m \in M]$, we have
	\[
	\mathcal{I}(\mathcal{V}(F))=\sqrt{ \langle \sigma^m(F) \mid m \in M \rangle }.
	\]
\end{theorem}

The following theorem shows that the condition $|k| > |M|$ in Theorem~\ref{thm:nullstellensatz_over_constants} cannot be omitted.

\begin{theorem} \label{theo:counter strong_in_main}
	There exists a finite set $F$ of difference equations over $\overline{\mathbb{Q}}$ such that
	\[
	\mathcal{I}(\mathcal{V}(F))\supsetneqq \sqrt{\langle \sigma^i(F) \mid i\in\mathbb{N}\rangle}.
	\]
\end{theorem}

\begin{remark}[Weak Nullstellensatz]
  Theorems~\ref{thm:nullstellensatz_over_constants} and~\ref{theo:counter strong_in_main} complement the weak Nullstellensatz from~\cite{difference_elimination} in a surprising way.
  Theorem~7.1 in~\cite{difference_elimination} established the weak Nullstellensatz for $M = \mathbb{N}$, that is,
  \[
    \mathcal{I}(\mathcal{V}(F))= \varnothing \iff 1 \in \sqrt{ \langle \sigma^m(F) \mid m \in M \rangle }
  \]
  without any restrictions on the cardinality of $k$. 
  However, the proof for the case of uncountable $k$ (see~\cite[Proposition~6.3]{difference_elimination}) was much simpler than the proof of the general statement.
  Our results indicate that this difference between the countable and uncountable cases is not an artifact of the proof in~\cite{difference_elimination} but rather a conceptual distinction.
\end{remark}

\begin{corollary}[Universality of the ring of sequences]
  Let $M$ be a monoid, let $k$ be an algebraically closed field such that $|k| > |M|$, and let $\bm{X} = (X_1, \ldots, X_n)$. 
	Then, for every subset $F$ of $k[\sigma^m(\bm{X}) \mid m \in M]$ and $g \in k[\sigma^m(\bm{X}) \mid m \in M]$ the following are equivalent:
	\begin{itemize}
	    \item $g=0$ holds for every solution of $F=0$ in any reduced $M$-ring containing $k$;
	    \item $g=0$ holds for every solution of $F = 0$ in $k^M$.
	\end{itemize}
\end{corollary}
\begin{proof}
    If the latter point holds, then $g^e\in \langle \sigma^m(F) \mid m \in M \rangle $ for some $e\geq 1$ by Theorem \ref{thm:nullstellensatz_over_constants}. Thus for every solution $\bm{x}$ in some reduced $M$-ring containing $k$ we have $g(\bm{x})^e=0$ and therefore $g(\bm{x})=0$ as desired.
\end{proof}

\begin{remark}[Nonconstant $k$]
Moreover, we prove a more general theorem (Theorem \ref{thm:nullstellensatz general}) than Theorem \ref{thm:nullstellensatz_over_constants} where the field $k$ is not necessarily constant.  We also establish an alternative formulation of the strong difference Nullstellensatz that works without any assumptions on the base difference field $k$ (Theorem~\ref{thm: alternative strong Nullstellensatz}). 
\end{remark}


\subsection{Undecidability results}

\begin{theorem}\label{thm:real}
  For every field $k$ such that $k \subseteq \mathbb{R}$ and every computable subfield $k_0 \subset k$,
  the following problem is undecidable: given a finite system of difference equations with coefficients in $k_0$, determine whether it has a solution in $k^{\mathbb{N}}$ (resp., $k^{\mathbb{Z}}$).
\end{theorem}

\begin{theorem}\label{thm:stong_undecidable}
  Let $M$ be $\mathbb{N}$ or $\mathbb{Z}$, let $k$ be a field of characteristic zero, and let $k_0 \subset k$ be a computable subfield.
  Then the following problem is undecidable: given a finite system of difference equations $F = 0$ and a difference equation $g = 0$ with coefficients in $k_0$, determine whether $g = 0$ holds for every solution on $F = 0$ in $k^M$.
\end{theorem}

\begin{corollary}\label{cor:strong_undecidability}
  Let $M$ be $\mathbb{N}$ or $\mathbb{Z}$, let $k$ be a field of characteristic zero, and let $k_0 \subset k$ be a computable subfield.
  Then the following problems are undecidable:
  \begin{enumerate}[label = \textbf{(P\arabic*)}]
      \item\label{prob:1ineq} Given $f_1, \ldots, f_{\ell}, g \in k_0[\sigma^m(\bm{X}) \mid m \in M]$ where $\bm{X} = (X_1, \ldots, X_n)$, determine whether the system $f_1 = \ldots = f_\ell = 0, g \neq 0$ has a solution in $k^{M}$.
      \item\label{prob:radical_ideal_membership} Given $f_1, \ldots, f_{\ell}, g \in k_0[\sigma^m(\bm{X}) \mid m \in M]$ where $\bm{X} = (X_1, \ldots, X_n)$, determine whether
      \[
        g \in \sqrt{ \langle \sigma^m(f_1), \ldots, \sigma^m(f_\ell) \mid m \in M \rangle }.
      \]
  \end{enumerate}
\end{corollary}

\begin{proposition}\label{prop:2d}
  Let $k$ be a field of characteristic zero and $k_0 \subset k$ be a computable subfield, and let the monoid $M$ be either $\mathbb{N}^2$ or $\mathbb{Z}^2$. 
  Then the following problem is undecidable: given a finite set $F$ of $M$-polynomials over $k_0$, decide whether the system $F = 0$ has a solution in $k^M$.
\end{proposition}

\begin{proposition}\label{prop:free_monoid}
    Let $k$ be a field of characteristic zero and $k_0 \subset k$ be a computable subfield, and let $M_2$ be a free monoid with two generators.
    Then the following problem is undecidable: given a finite set $F$ of $M_2$-polynomials over $k_0$, decide whether $F = 0$ has a solution in $k^{M_2}$.
\end{proposition}


\section{Proofs}

Throughout this section, we will use the following notation.
For a tuple of sequences $(\{x_{1, i}\}_{i \in M}, \ldots, \{x_{n, i}\}_{i \in M})$, we will denote $\bm{x}_i = (x_{1, i}, \ldots, x_{n, i})$ for every $i \in M$, and the original tuple of sequences will be denoted by $\{\bm{x}_i\}_{i \in M}$.


\subsection{Proof of Theorem~\ref{thm:nullstellensatz_over_constants}}

In this section we establish two closely related versions of a strong difference Nullstellensatz (Theorem~\ref{thm:nullstellensatz general} and Theorem \ref{thm: alternative strong Nullstellensatz}). Theorem \ref{thm:nullstellensatz general} contains Theorem \ref{thm:nullstellensatz_over_constants} as a special case.

We begin by introducing the notation necessary to state our general result. Let $M$ be a monoid and let $k$ be an $M$-field. We note that for any field extension $K$ of $k$ the map $k\to K^M,\ a\mapsto (\sigma^m(a))_{m\in M}$ is a morphism of $M$-rings. Let $\mathbf{X} = (X_1, \ldots, X_n)$.
As in Section~\ref{subsec:universality}, for a subset $F$ of $k[\sigma^m(\mathbf{X})\mid m\in M]$, we set
\[
\mathcal{V}(F)=\{\bm{x}\in (k^M)^n \mid f(\bm{x})=0\ \forall \ f\in F\},
\]
and for a subset $S$ of $(k^M)^n$ we set
\[
\mathcal{I}(S)=\{f\in k[\sigma^m(\mathbf{X})\mid m\in M] \mid f(\bm{x})=0\ \forall \ \bm{x}\in S \}.
\]
\begin{theorem}[Strong Nullstellensatz]\label{thm:nullstellensatz general}
	Let $k$ be an algebraically closed $M$-field such that $|k| > |M|$. Then, for every subset $F$ of $k[\sigma^m(\mathbf{X})\mid m\in M]$ we have
	\[
	\mathcal{I}(\mathcal{V}(F))=\sqrt{\langle \sigma^m(F) \mid m \in M \rangle}.
	\]
\end{theorem}

In Section \ref{subsec:Counterexample} we present an example that shows that the assumption $|k|>|M|$ in Theorem \ref{thm:nullstellensatz general} cannot be omitted. However, we also have an alternative formulation of Theorem \ref{thm:nullstellensatz general} that works without any assumptions on the base difference field $k$. For a subset $F$ of $k[\sigma^m(\mathbf{X})\mid m\in M]$ we set 
\[
\mathfrak{I}(F)=\{f\in k[\sigma^m(\mathbf{X})\mid m\in M] \mid \ \text{for every field extension $K/k$, $f$ vanishes on all solutions of $F$ in $K^M$} \}
\]

\begin{theorem} \label{thm: alternative strong Nullstellensatz}
Let $k$ be an $M$-field and $F\subseteq k[\sigma^m(\mathbf{X})\mid m\in M]$. Then
\[
  \mathfrak{I}(F)=\sqrt{\langle \sigma^m(F) \mid m \in M\rangle}.
\]
\end{theorem}

For the proofs of Theorems \ref{thm:nullstellensatz general} and \ref{thm: alternative strong Nullstellensatz} we will need the following version of the strong algebraic Nullstellensatz for polynomials in infinitely many variables. 
Let $k$ be a field and $\bf{Y}$ a (not necessarily finite) set of indeterminates over $k$. For $F\subseteq k[\mathbf{Y}]$ we set 
\[
  \mathbb{V}(F) = \{\mathbf{y}\in k^{\mathbf{Y}} \mid f(\mathbf{y})=0\ \forall \ f\in F\},
\] 
and for $S\subseteq k^\mathbf{Y}$ we set 
\[
\mathbb{I}(S) = \{f\in k[\mathbf{Y}] \mid f(\mathbf{y})=0\ \forall \ \mathbf{y}\in S\}.
\]

\begin{lemma} \label{lemma: algebraic strong Nullstellensatz}
Let $k$ be an algebraically closed field and $F\subseteq k[\mathbf{Y}]$. If $|k|>|\mathbf{Y}|$, then $\mathbb{I}(\mathbb{V}(F))=\sqrt{\langle F\rangle}$.
\end{lemma}
\begin{proof}
This follows from the main theorem of \cite{Lang:Nullstellensatz}.
\end{proof}

\begin{proof}[Proof of Theorem \ref{thm:nullstellensatz general}]
	As $\mathcal{I}(S)$ is a radical $M$-invariant ideal, for any subset $S$ of $k^M$, we have 
	\[
	\sqrt{\langle \sigma^m(F) \mid m \in M\rangle }\subseteq \mathcal{I}(\mathcal{V}(F)).
	\]
	
	To establish the reverse inclusion we set $\mathbf{Y}=\{\sigma^m(\mathbf{X}) \mid m\in M\}$, so that $(k^M)^n$ can be identified with $k^{\mathbf{Y}}$. The nature of the map $k\to k^M,\ a\mapsto (\sigma^m(a))_{m\in M}$ is such that for $f\in k[\sigma^m(\mathbf{X})\mid m\in M]$ and $\mathbf{x}\in (k^M)^n$ we have $f(\bm{x})=0\in (k^M)^n$ if and only if $\sigma^m(f)(\bm{x})=0\in k$ for all $m\in M$. 
	So, under the identification $(k^M)^n=k^{\bm{Y}}$, we have $\mathcal{V}(I)=\mathbb{V}(I)$ for any $M$-invariant ideal $I$ of $k[\sigma^m(\mathbf{X})\mid m\in M] = k[\bm{Y}]$.
	Similarly, for any subset $S$ of $(k^M)^n=k^{\bf{Y}}$ we have $f\in\mathcal{I}(S)\subseteq k[\sigma^m(\mathbf{X})\mid m\in M]$ if and only if $\sigma^m(f)\in \mathbb{I}(S)\subseteq k[\bf{Y}]$ for all $m\in M$, in particular, $\mathcal{I}(S)\subseteq \mathbb{I}(S)$. 
	Clearly $\mathcal{V}(F)=\mathcal{V}(I)$, where $I = \langle \sigma^m(F) \mid m \in M\rangle$, and so	
	\[
	\mathcal{I}(\mathcal{V}(F))=\mathcal{I}(\mathcal{V}(I))=\mathcal{I}(\mathbb{V}(I))\subseteq\mathbb{I}(\mathbb{V}(I)=\sqrt{I}.
	\]
	In the case that $M$ is infinite the last equality here follows from Lemma \ref{lemma: algebraic strong Nullstellensatz} since then $|X|=n|M|=|M|<|k|$. In the case that $M$ is finite, the last equality reduces to the usual algebraic strong Nullstellensatz.
\end{proof}

\begin{proof}[Proof of Theorem \ref{thm: alternative strong Nullstellensatz}]
Again, the inclusion $\sqrt{I}\subseteq\mathfrak{I}(F)$, where $I = \langle \sigma^m(F) \mid m \in M \rangle$, is clear. 
To establish the reverse inclusion we let $K$ denote an algebraically closed field extension of $k$ with $|K|>|M|$ and we proceed similarly to the proof of Theorem \ref{thm:nullstellensatz general}: For $\bm{Y}=\{\sigma^m(\mathbf{X}) \mid m\in M\}$ we have, under the identification $(K^M)^n=K^{\bm{Y}}$, that 
\[
  \{\bm{x}\in (K^M)^n \mid f(\bm{x})=0 \ \forall \ f\in F\} = \{\bm{x}\in K^{\bm{Y}} \mid \sigma^m(f)(\bm{x})=0 \ \forall \ f\in F, \ m\in M\}.
\]
Thus, if $f\in\mathfrak{I}(F)\subset k[\sigma^m(\mathbf{X})\mid m\in M] = k[\bm{Y}]$, then $f\in \mathbb{I}(\mathbb{V}(I))$. Note that here $I \subseteq k[\sigma^m(\mathbf{X})\mid m\in M] \subseteq K[\bm{Y}]$ but $\mathbb{I}$ and $\mathbb{V}$ are applied with respect to $K$. So it follows from Lemma \ref{lemma: algebraic strong Nullstellensatz} that $f\in \sqrt{\langle I \rangle}$, where $\langle I \rangle \subseteq K[X]$.
But $K[X] = k[X]\otimes_k K$ and $\langle I \rangle = I \otimes_k K$. Therefore, if $e\geq 1$ is such that $f^e\in  \langle I \rangle = I \otimes_k K$, then $f^e\in (I\otimes_k K)\cap k[X]= I$. 
Thus $f\in \sqrt{I}$ as desired.
\end{proof}


\subsection{Proof of Theorem~\ref{thm:real}}

Let $M$ be $\mathbb{N}$ or $\mathbb{Z}$.
For every polynomial equation $P(t_1, \ldots, t_n) = 0$ with coefficients in $\mathbb{Z}$, we will construct a system of difference equations $F_P = 0$ over $\mathbb{Q}$ such that $P = 0$ has a solution in $\mathbb{Z}^n$ if and only if $F_P = 0$ has a solution in $k^{M}$.
Then the theorem will follow from the undecidability of diophantine equations~\cite{M70}.
   
   \begin{lemma}\label{lem:inf_zeros}
     Let $\bm{Y} = (Y_1, \ldots, Y_6)$.
     There exists a finite set $G \subset \mathbb{Q}[\sigma^i(X), \sigma^i(\bm{Y}) \mid i \in M ]$ such that, for every solution of $G = 0$ in $k^{M}$, the sequence $(x_i)_{i \in M}$ corresponding to $X$ has the property that $(x_i)_{i \in \mathbb{N}}$ contains infinitely many zeroes.
     
     Moreover, for every sequence $(x_i)_{i\in M} \in k^{M}$ such that $(x_i)_{i \in \mathbb{N}}$ contains infinitely many zeroes, there exists a solution of $G = 0$ in $k^M$ such that $(x_i)_{i\in M}$ is the $X$-coordinate of the solution.
   \end{lemma}
   
   \begin{proof}
      We define $G$ as
      \[
      G := \{ XY_1, Y_2 - Y_3^2 - Y_4^2 - Y_5^2 - Y_6^2, \sigma(Y_2) - Y_2 + 1 - Y_1\}.
      \]
      Consider a solution 
      \[
      \bigl( (x_i)_{i \in M}, (y_{1, i})_{i \in M}, \ldots, (y_{6, i})_{i \in M} \bigr) \text{ of } G = 0 \text{ in } k^{M}.
      \]
      If $(x_i)_{i \in \mathbb{N}}$ contains only finitely many zeroes, then $(y_{1, i})_{i \in \mathbb{N}}$ contains only finitely many nonzero elements.
      In other words, there exists $N \in \mathbb{N}$ such that $y_{1, i} = 0$ for every $i > N$.
      Thus, $y_{2, i + 1} = y_{2, i} - 1$ for every $i > N$, so there exists $i_0$ such that $y_{2, i_0} < 0$.
      This contradicts the fact that $y_{2, i_0} = y_{3, i_0}^2 + y_{4, i_0}^2 + y_{5, i_0}^2 + y_{6, i_0}^2 \geqslant 0$.
      
      To prove the second claim of the lemma, consider a sequence $(x_i)_{i \in M}$ such that $(x_i)_{i \in \mathbb{N}}$ contains infinitely many zeroes.
      We will construct a corresponding solution of $G = 0$ in $k^{M}$.
      Consider positive integers $i_1 < i_2 < i_3 < \ldots$ such that $x_{i_n} = 0$ for every $n > 0$.
      Then we set
      \[
      y_{1, j} = 
      \begin{cases}
         i_{m + 1} - i_m, \text{ if } j = i_m \text{ for some } m,\\
         0, \text{ otherwise}
      \end{cases} 
      \quad\text{ and }\quad
      y_{2, j} = 
      \begin{cases}
         i_{m + 1} - j, \text{ if } i_m < j \leqslant i_{m + 1} \text{ for some } m,\\
         i_1 - j, \text{ otherwise}.
      \end{cases}
      \]
      The choice of $i_1, i_2, \ldots$ implies that $x_j y_{1, j} = 0$ for all $j \in M$.
      A direct computation shows that $y_{2, j + 1} = y_{2, j} - 1 + y_{1, j}$ for all $j \in M$.
      Finally, the existence of $y_{3, j}, y_{4, j}, y_{5, j}, y_{6, j}$ satisfying $y_{2, j} = y_{3, j}^2 + y_{4, j}^2 + y_{5, j}^2 + y_{6, j}^2$ follows from the fact that $y_{2, j}$ is a nonnegative integer and Lagrange's four-square theorem~\cite[Theorem 369]{NT}.
   \end{proof}
   
   We return to the proof of Theorem~\ref{thm:real}.
   We apply Lemma~\ref{lem:inf_zeros} $n + 1$ times, and obtain $n + 1$ systems $G_0 = 0, \ldots, G_{n} = 0$ with distinguished unknowns $X_0, \ldots, X_{n}$.
   We set
   \[
   F_P := G_0 \cup \ldots \cup G_{n} \cup \{ X_0 - P(X_1, \ldots, X_n), (\sigma(X_1) - X_1)^2-1, \ldots, (\sigma(X_n) - X_n)^2-1\}.
   \]
   We will show that $F_P = 0$ has a solution in $k^{M}$ if and only if $P(t_1, \ldots, t_n) = 0$ has a solution in $\mathbb{Z}$.

   \paragraph{Solution of $F_P = 0$ $\implies$ solution of $P = 0$.}
   Consider a solution of $F_P$ in $k^{M}$.
   For every $0 \leqslant m \leqslant n$, we denote the $X_m$-coordinate of the solution by $(x_{m, i})_{i \in M}$.
   For every $1 \leqslant m \leqslant n$, the sequence $(x_{m, i})_{i \in M}$ contains infinitely many zeroes due to Lemma~\ref{lem:inf_zeros}, every two consecutive numbers in the sequence differ by one, thus all the numbers in the sequence are integers.
   Since $(x_{0, i})_{i \in \mathbb{N}}$ contains infinitely many zeroes, the diophantine equation $P(t_1, \ldots, t_n) = 0$ has an integer solution.
   
   \paragraph{Solution of $P = 0$ $\implies$ solution of $F_P = 0$.}
   Consider a solution $(a_1, \ldots, a_m)$ of $P(t_1, \ldots, t_m) = 0$ in $\mathbb{Z}^n$.
   Consider sequences $(x_{1, i})_{i \in M}, \ldots, (x_{n, i})_{i \in M}$ such that
   \begin{itemize}
       \item every two consecutive numbers in the sequences differ by one;
       \item for every $1 \leqslant m \leqslant n$, $(x_{m, i})_{i = 0}^\infty$ contains infinitely many zeros;
       \item $x_{1, i} = a_1, \ldots, x_{n, i} = a_n$ for infinitely many $i$.
   \end{itemize}
   We define $x_{0, i}$ as $P(x_{1, i}, \ldots, x_{n, i})$ for every $i \in M$ and observe that $(x_{0, i})_{i \in \mathbb{N}}$ contains infinitely many zeroes.
   The defined sequences satisfy equations 
   \[
   X_0 - P(X_1, \ldots, X_n) = (\sigma(X_1) - X_1)^2-1 =  \ldots = (\sigma(X_n) - X_n)^2-1 = 0.
   \]
   The second part of Lemma~\ref{lem:inf_zeros} implies that, for every $0 \leqslant m \leqslant n$, the sequence $(x_{m, i})_{i \in M}$ can be extended to a solution of $G_m = 0$.
   Thus, we obtain a solution of $F_P = 0$.


\subsection{Proofs of Theorem~\ref{thm:stong_undecidable} and Corollary~\ref{cor:strong_undecidability}}

We will first establish a lemma that draws a connection between the strong difference Nullstellensatz and iterations of piecewise polynomial maps. This lemma is crucial for the proof of Theorem \ref{thm:stong_undecidable} and for establishing the counterexample in Theorem \ref{theo:counter strong_in_main}.

Let $k$ be a field. For a subset $F$ of $k[\mathbf{X}]=k[X_1,\ldots,X_n]$ we denote the closed subset of $\mathbb{A}^n_k$ defined by $F$ with $V(F)$. Recall that a subset $V$ of $\mathbb{A}^n_k$ is locally closed if it is of the form $V(F)\smallsetminus V(F')$ for subsets $F$ and $F'$ of $k[\mathbf{X}]$. A regular function $f\colon V\to\mathbb{A}^1_k$ on $V$ is a \emph{polynomial function} if it is the restriction of a regular function $\mathbb{A}^n_k\to \mathbb{A}^1_k$, i.e., if it is given by a polynomial in $k[\mathbf{X}]$.

\begin{definition}
  A \emph{piecewise polynomial function} $\mathbb{A}^n_k \to \mathbb{A}^1_k$ is a partition of $\mathbb{A}^n_k$ into locally closed subsets $C_1,\ldots,C_m$, together with a polynomial function $f_i$ on every $C_i$.
  
  A \emph{piecewise polynomial map} $\mathbf{p} \colon \mathbb{A}^n_k\to\mathbb{A}^n_k$ is an $n$-tuple $(p_1,\ldots,p_n)$ of piecewise polynomial functions.
\end{definition}

Note that a piecewise polynomial map $\mathbf{p} \colon \mathbb{A}^n_k\to\mathbb{A}^n_k$ defines an actual map $\mathbb{A}^n_k(K)\to  \mathbb{A}^n_k(K)$ for every field extension $K$ of $k$.


\begin{lemma} \label{lemma: equivalence iteration and strong Nullstellensatz}
  Let $M$ be $\mathbb{N}$ or $\mathbb{Z}$.
  Let $\mathbf{p}\colon \mathbb{A}_k^n\to\mathbb{A}_k^n$ be a piecewise polynomial map and let $V$ be a closed subset of $\mathbb{A}_k^n$. Then there exist (and can be computed algorithmically) an integer $r\geq 1$ and difference polynomials $f_1, \ldots, f_\ell, g\in k[\sigma^i(T_1),\ldots,\sigma^i(T_r)|\ i\in\mathbb{N}]$ such that for every field extension $K$ of $k$ the following two statements are equivalent:
  \begin{itemize}
      \item There exists a sequence $(\mathbf{x}_i)_{i\in\mathbb{N}}=(x_{1, i}, \ldots, x_{n, i})_{i \in \mathbb{N}}\in (K^\mathbb{N})^n$ such that
      \[
        \mathbf{x}_0 \in V(K) , \quad \mathbf{x}_{i + 1} = \mathbf{p}(\mathbf{x}_i) \text{ for every } i \in \mathbb{N},
      \]
     and $x_{n,i}
     \neq 0$ for $i\geq 1$.
      
      \item There exists a solution of $f_1 = \ldots = f_\ell = 0$ in $(K^M)^r$ such that $g$ does not vanish on this solution.
  \end{itemize}
\end{lemma}

Before showing the construction of the systems of difference equations in full generality, we will illustrate it on two examples.

\begin{example}\label{ex:piecewise_simple}
We will use the notation of Lemma~\ref{lemma: equivalence iteration and strong Nullstellensatz}. Let 
\[
  k = \mathbb{C}, \quad n = 1,\quad p(x) = x + 1, \quad V = \{ 0 \}.
\]
We introduce two difference variables $X$ and $U'$, and consider difference polynomials
\[
  \widetilde{f}_1 := \sigma(X) - p(X) = \sigma(X) - (X + 1), \quad \widetilde{f}_2 := XU' - 1.
\]
Every sequence $(x_i)_{i \in \mathbb{N}}$ satisfying $\widetilde{f}_1 = 0$ obeys the recurrence $x_{i + 1} = p(x_i)$.
Furthermore, such a sequence can be extended to a solution of $\widetilde{f}_1 = \widetilde{f}_2 = 0$ if and only if $x_i \neq 0$ for every $i \geqslant 0$ (compare with $i \geqslant 1$ in the statement of the lemma).

Now we would like to force $(x_i)_{i \in \mathbb{N}}$ to have at least one term in $V$.
For doing this, we will introduce one more difference variable $U$ and difference polynomials
\[
  f_3 := U (U - 1), \quad f_4 := (\sigma(U) - U) (\sigma(U) - U - 1), \quad g := \sigma(U) - U.
\]
Consider a sequence $(u_i)_{i \in \mathbb{N}}$ which is a solution of $f_3 = f_4 = 0, \; g \neq 0$.
The equations $f_3 = f_4 = 0$ imply that $(u_i)_{i \in \mathbb{N}}$ is a ``step sequence'' in the sense that it takes only values zero and one and each next value is the same or greater by one.
There are three types of sequences that satisfy these conditions
\[
   (0, \ldots, 0, 1, 1, \ldots), \quad (0, 0, 0, 0, \ldots), \quad (1, 1, 1, 1, \ldots).
\]
The two last are ruled out by the extra condition $g \neq 0$.

We introduce new polynomials
\[
  f_5 := (\sigma(U) - U)X, \quad f_1 := \sigma(U) \widetilde{f}_1 = \sigma(U) (\sigma(X) - (X + 1)), \quad f_2 = U\widetilde{f}_2 = U(XU' - 1).
\]
Consider a triple of sequences $(x_i, u_i, u_i')_{i \in \mathbb{N}}$ satisfying 
\begin{equation}\label{eq:ex_system}
f_1 = f_2 = f_3 = f_4 = f_5 = 0, \quad g \neq 0.
\end{equation}
As we have shown, there will be $i_0 \in \mathbb{N}$ such that 
\[
  u_0 = \ldots = u_{i_0} = 0 \quad \text{ and } 1 = u_{i_0 + 1} = u_{i_0 + 2} = \ldots.
\]
Equation $f_5 = 0$ ensures that $x_{i_0} \in V$.
The fact that we have multiplied $\widetilde{f}_1$ and $\widetilde{f}_2$ by $\sigma(U)$ and $U$, respectively, implies that $\widetilde{f}_1$ and $\widetilde{f}_2$ have to vanish on the indices $i \geqslant i_0$ and $i > i_0$, respectively.

To summarize, we see that the sequence $(y_i)_{i \in \mathbb{N}} = (x_{i_0 + i})_{i \in \mathbb{N}}$ satisfies $y_{i + 1} = p(y_i)$, $y_0 \in V$, and $y_i \neq 0$ for $i \geqslant 1$.
On the other hand, any such sequence can be completed by $u = (0, 1, 1, \ldots)$ and $u' = (0, 1 / y_1, 1 / y_2, \ldots)$ to a solution of~\eqref{eq:ex_system}.
\end{example}

\begin{example}
Now we consider a version of Example~\ref{ex:piecewise_simple} where $p$ is actually a piecewise polynomial function, not just polynomial.
Let
\[
  k = \mathbb{C}, \quad n = 1,\quad p(x) = \begin{cases}
  x + 1, \text{ if } x\neq 2\\
  1, \text{ if } x = 2,
  \end{cases}
  \quad V = \{ 0 \}.
\]
We define $C_1 := \mathbb{A}^1 \setminus \{2\}$, $C_2 := \{2\}$, $q_1(x) := x + 1$, and $q_2(x) := 1$ so that $p|_{C_1} = q_1$ and $p|_{C_2} = q_2$.
Our strategy would be to define an indicator sequence that will tell us whether $x_i \in C_2$ or not.
For doing this, we introduce two difference variables $Y$ and $Z$ and difference polynomials 
\[
  f_6 := Z(X - 2), \quad f_7 := Z + Y(X - 2) - 1.
\]
Consider any tuple of sequences $(x_i, y_i, z_i)_{i\in\mathbb{N}}$ satisfying $f_6 = f_7 = 0$.
Whenever $x_i \not\in C_2$, $f_6 = 0$ implies that $z_i = 0$.
If $x_i \in C_2$, then $f_7 = 0$ implies that $z_i = 1$.
Thus $z_i$ is an indicator for $x_i \in C_2$.
Therefore, we have
\begin{equation}\label{eq:p}
  p(x_i) = (1 - z_i)(x_i + 1) + z_1 \cdot 1 \text{ for every } i\in \mathbb{N}.
\end{equation}
We can now adapt the system~\eqref{eq:ex_system} from Example~\ref{ex:piecewise_simple} as follows.
We take the same $f_2, f_3, f_4, f_5$, but change $f_1$ to be
\[
  f_1 := \sigma(U) (X - (1 - Z)(X + 1) - Z)
\]
according to~\eqref{eq:p}.
Then, combining the argument from Example~\ref{ex:piecewise_simple} and this example, one can see that any sequence $(x_i)_{i \in \mathbb{N}}$ with 
\begin{equation}\label{eq:component_condition}
  x_0 = 0, \quad x_{i + i} = p(x_i), \quad \text{ and } \quad x_i \neq 0 \text{ for } i \geqslant 1
\end{equation}
can be extended to a solution of
\[
  f_1 = f_2 = \ldots = f_7 = 0, \quad g \neq 0.
\]
On the other hand, for every solution for the above system of difference equations, the $X$-component satisfies~\eqref{eq:component_condition} after removing several first terms.
\end{example}

\begin{proof}[Proof of Lemma~\ref{lemma: equivalence iteration and strong Nullstellensatz}]
	Let $\mathbf{p} = (p_1,\ldots,p_n)$. Since finite intersections of locally closed subsets are locally closed, we can find a partition $C_1,\ldots,C_m$ of $\mathbb{A}^n_k$ that works for every $p_i$. 
	For $j = 1,\ldots,m$ let $\mathbf{q}_j = (q_{j,1},\ldots,q_{j,n})\in k[\mathbf{X}]^n$ be such that $\mathbf{p}(a)=\mathbf{q}_j(a)$ for all $a\in C_j(K)$ and all field extensions $K$ of $k$.
	
	For every closed subset $W$ of $\mathbb{A}_k^n$ we define a polynomial system $S_W$ as follows. 
	Let $h_1, \ldots, h_t \in k[\mathbf{X}]$ be polynomials such that $W = V(h_1, \ldots, h_t)$. 
	Let $S_W=S_W(\mathbf{X},\mathbf{Y},Z)$ be the system in the variables $\mathbf{X}=(X_1,\ldots,X_n)$, $\mathbf{Y}=(Y_1,\ldots,Y_t)$ and $Z$ given by
	$$ Z h_1(\mathbf{X}),\ldots, Zh_t(\mathbf{X}),\ Z+Y_1h_1(\mathbf{X})+\ldots+Y_th_t(\mathbf{X})-1.$$
	Note that for a field extension $K$ of $k$ and a solution $(\mathbf{x},\mathbf{y},z)\in K^{n + t + 1}$  
	we have $z = 1$ if $\mathbf{x} \in W$ and $z = 0$ if $\mathbf{x} \notin W$. 
	Moreover, for every field extension $K$ of $k$ and $\mathbf{x}\in K^n$, there exist $\mathbf{y}\in K^t$ and $z\in K$ such that $(\mathbf{x},\mathbf{y},z)$ is a solution of $S_W$.

	Now for every $j=1,\ldots,m$ write $C_j=W_j\smallsetminus W'_j$, where $W_j, W_j'$ are closed subsets of $\mathbb{A}^n_k$ with $W_j'\subseteq W_j$ and consider the systems $S_j=S_{W_j}=S_{W_j}(\mathbf{X},\mathbf{Y}_j,Z_j)$ and $S'_j=S_{W'_j}=S_{W'_j}(\mathbf{X},\mathbf{Y}'_j,Z'_j)$.  Let $g_1,\ldots,g_s\in k[\mathbf{X}]$ be such that $V(g_1,\ldots,g_s)=V$.
	
	Let $S$ denote the system of difference equations in the variables 
	\[
	U,U',\mathbf{X},\mathbf{Y}_1,\ldots,\mathbf{Y}_m, Z_1,\ldots,Z_m,\mathbf{Y}'_1,\ldots,\mathbf{Y}'_m, Z'_1,\ldots,Z'_m
	\]
	given by
	\[
	S_1(\mathbf{X},\mathbf{Y}_1,Z_1),\;\ldots,\;S_m(\mathbf{X},\mathbf{Y}_m,Z_m),\;S'_1(\mathbf{X},\mathbf{Y}'_1,Z'_1),\;\ldots,\; S_m(\mathbf{X},\mathbf{Y}'_m,Z'_m),
	\]
	\[
	 \sigma(U)(\sigma(\mathbf{X}) - (\mathbf{q}_1(\mathbf{X})(Z_1-Z_1') + \ldots + \mathbf{q}_m(\mathbf{X})(Z_m-Z_m'))),
	 \]
	 \[
	U(U - 1),\; \; (\sigma(U) - U)(\sigma(U) - U - 1),
	\]
	\[
    U(X_nU'-1),\;\; (\sigma(U)-U)g_1(\mathbf{X}),\;\ldots,\;(\sigma(U)-U)g_s(\mathbf{X}).
    \]
	
	We will show that $S=\{f_1,\ldots,f_\ell\}$ and $g=\sigma(U)-U$ have the property of the lemma. 
	To this end, let us fix a field extension $K$ of $k$ and let us first assume that $$a=(u_i,u'_i,\mathbf{x}_i,\mathbf{y}_{1,i},\ldots,\mathbf{y}_{m,i},z_{1,i},\ldots,z_{m,i},\mathbf{y}'_{1,i},\ldots,\mathbf{y}'_{m,i},z'_{1,i},\ldots,z'_{m,i})_{i\in M}\in (K^M)^r$$ is a solution of $S$ such that $\sigma(U)-U$ does not vanish on $a$. 
	We observe that the equations $U(U - 1) = 0$ and  $(\sigma(U) - U)(\sigma(U) - U - 1) = 0$ imply that either $u_i=0$ for all $i$, $u_i=1$ for all $i$ or, there exists an $i_0\in M$, such that 
	$$u_i=\begin{cases}
	0 & \text{ for } i\leq i_0, \\
	1 & \text{ for } i> i_0
	\end{cases}.$$
	Since $\sigma(U)-U$ does not vanish on $a$, the sequence $(u_i)_{i\in M}$ is of the latter kind. 
	The equations $(\sigma(U)-U)g_1(\mathbf{X}) = \ldots = (\sigma(U)-U)g_s(\mathbf{X}) = 0$ imply that $g_1(\mathbf{x}_{i_0}) = \ldots = g_s(\mathbf{x}_{i_0})=0$, i.e., $\mathbf{x}_{i_0}\in V(K)$.
	
	For every $j=1,\ldots,m$ and $i\in M$, we have
	$$	z_{j,i}=\begin{cases} 1 \text{ if } \mathbf{x}_i\in W_j(K), \\
	0 \text{ if } \mathbf{x}_i\notin W_j(K).

	\end{cases}
	$$
	Similarly, 	$$	z'_{j,i}=\begin{cases} 1 \text{ if } \mathbf{x}_i\in W'_j(K), \\
	0 \text{ if } \mathbf{x}_i\notin W'_j(K).
	
	\end{cases}
	$$
	Therefore
		$$	z_{j,i}-z'_{j,i}=\begin{cases} 1 \text{ if } \mathbf{x}_i\in C_j(K), \\
	0 \text{ if } \mathbf{x}_i\notin C_j(K).
	
	\end{cases}
	$$
	 Thus the equations $\sigma(U)(\sigma(\mathbf{X})-(\mathbf{q}_1(\mathbf{X})(Z_1-Z_1')+\ldots+\mathbf{q}_m(\mathbf{X})(Z_m-Z_m'))) = 0$ show that $\mathbf{x}_{i+1}=\mathbf{p}(\mathbf{x}_i)$ for all $i\geq i_0$. 
	 Finally, the equation $U(U'X_n-1) = 0$ shows that $x_{n,i}\neq 0$ for $i> i_0$. 
	 Therefore the sequence $(\mathbf{x}_{i_0+i})_{i\in\mathbb{N}}$ has the desired properties.

	 \medskip
	 
	 Conversely, let us assume that the sequence $(\mathbf{x}_i)_{i\in\mathbb{N}}$ satisfies $\mathbf{x}_0\in V(K),\ \mathbf{x}_{i + 1} = \mathbf{p}(\mathbf{x}_i)$ for $i \in \mathbb{N}$ and $x_{n,i}\neq 0$ for $i\geq 1$. 
	 We extend this sequence to a solution 
	 $$a=(u_i,u'_i,\mathbf{x}_i,\mathbf{y}_{1,i},\ldots,\mathbf{y}_{m,i},z_{1,i},\ldots,z_{m,i},\mathbf{y}'_{1,i},\ldots,\mathbf{y}'_{m,i},z'_{1,i},\ldots,z'_{m,i})_{i\in M}\in (K^M)^r$$ of $S$
	 such that $g$ does not vanish at $a$. For $M=\mathbb{Z}$ we set $x_{j,i}=0$ for $i<0$ and $j=1,\ldots,m$.
	 We define 
	 $$u_i=\begin{cases} 1 \text{ for } i\geq 1, \\
	 0 \text{ otherwise}
	 \end{cases} \text{ and } u'_i=\begin{cases} \frac{1}{x_{n,i}} \text{ for } i\geq 1, \\
	 0 \text{ otherwise.}
	 \end{cases} $$
	 For $i\in M$ we choose $\mathbf{y}_{j,i}\in K^{s_j}$ and $z_{j,i}\in K$ such that $(\mathbf{x}_i,\mathbf{y}_{j,i},z_{j,i})$ is a solution of $S_j(\mathbf{X},\mathbf{Y}_j,Z_j)$. Similarly, 
	  we choose $\mathbf{y}'_{j,i}\in K^{s'_j}$ and $z'_{j,i}\in K$ such that $(\mathbf{x}_i,\mathbf{y}'_{j,i},z'_{j,i})$ is a solution of $S'_j(\mathbf{X},\mathbf{Y}'_j,Z'_j)$.
	  Then $a$ is a solution of $S$ such that $g$ does not vanish at $a$.
\end{proof}

We will need one more preparatory lemma for the proof of Theorem \ref{thm:stong_undecidable}. For every $n$, by $T_n$ we will denote the sequence of all nondecreasing $n$-tuples of nonnegative integers listed in ascending colexicographic order.
For example,
\[
T_1 = ((0), (1), (2), (3), \ldots) \quad\text{ and }\quad T_2 = ((0, 0),\; (0, 1),\; (1, 1),\; (0, 2),\; (1, 2),\; (2, 2),\ldots).
\]

\begin{lemma} \label{lemma: get Tn}
	For every $n\geq 1$, there exists a piecewise polynomial map $p\colon\mathbb{A}^{n}_{\mathbb{Q}}\to\mathbb{A}^{n}_{\mathbb{Q}}$
	such that for the sequence $(\bm{x}_i)_{i\in \mathbb{N}}=(x_{1,i},\ldots,x_{n,i})_{i\in\mathbb{N}}$ defined by
		\[
	\mathbf{x}_0=(0,\ldots,0) \ \;\&\; \ \mathbf{x}_{i + 1} = \mathbf{p}(\mathbf{x}_i) \text{ for all } i\in \mathbb{N},
	\]
	we have  $(\mathbf{x}_i)_{i \in \mathbb{N}} = T_n$.
\end{lemma}
\begin{proof}
	The successor of a nondecreasing $n$-tuple $(a_1,\ldots,a_n)\in\mathbb{N}^n$ in $T_n$ is $(a_1,\ldots,a_{r-1},a_r+1,a_{r+1},\ldots,a_n)$ if there exists an $r$ with $1\leq r< n $ such that $a_1=\ldots=a_r\neq a_{r+1}$ and $(0,\ldots,0,a_n+1)$ if there exists no such $r$, i.e., if $a_1=\ldots=a_n$. Thus, the piecewise polynomial map $\mathbf{p} = (p_1,\ldots,p_n)$ defined by
	$$p_i(x_1,\ldots,x_n)=\begin{cases} x_i+1 \text{ if } x_1=\ldots=x_i\neq x_{i+1}, \\
	0 \text{ if } x_1=\ldots=x_n, \\
	x_i \text{ otherwise}, \end{cases}$$
	for $i=1,\ldots,n-1$ and
	$$p_n(x_1,\ldots,x_n)=\begin{cases} x_n+1 \text{ if } x_1=\ldots=x_n, \\
	x_n \text{ otherwise},
	\end{cases}
	$$
	has the desired property.
\end{proof}

\begin{proof}[Proof of Theorem \ref{thm:stong_undecidable}]
We will prove Theorem~\ref{thm:stong_undecidable} by showing that the decidability of the problem of Theorem~\ref{thm:stong_undecidable}  implies the decidability of Hilbert's tenth problem for the integers.
Let $P \in \mathbb{Z}[t_1, \ldots, t_n]$ with $P(0,\ldots,0)\neq 0$ and consider the piecewise polynomial map $\mathbf{q}\colon \mathbb{A}^m_{\mathbb{Q}}\to\mathbb{A}^m_\mathbb{Q}$, where $m=n\cdot n!+1$, defined as follows: 
thinking of $\mathbb{A}^m_\mathbb{Q}$ as $(\prod_{\pi\in S_n}\mathbb{A}^n_\mathbb{Q})\times \mathbb{A}^1_\mathbb{Q}$ we write $\bm{x}=((\bm{x}_{\pi})_{\pi\in S_n},x_{r})$, where each $\bm{x}_{\pi}$ is an $n$-tuple. 
We set 
\[
  \mathbf{q}(\bm{x}) = \left((\mathbf{p}_\pi(x_\pi))_{\pi\in S_n},\prod_{\pi\in S_n}P(\bm{x}_\pi)\right),
  \]
where $\mathbf{p}_\pi\colon\mathbb{A}^n_\mathbb{Q}\to\mathbb{A}^n_\mathbb{Q}$ is the map $\mathbf{p} \colon\mathbb{A}^n_\mathbb{Q}\to\mathbb{A}^n_\mathbb{Q}$ from Lemma \ref{lemma: get Tn} but conjugated with the permutation $\pi$. 
So, if we define $(\bm{x}_i)_{i\in\mathbb{N}}\in (\mathbb{Q}^\mathbb{N})^m$ by 
$\bm{x}_0=(0,\ldots,0)$ and $\bm{x}_{i+1}= \mathbf{q}(\bm{x}_i)$ for $i\geq 0$, 
we see that, for every element $a$ of $\mathbb{N}^n$, there exist $i\in\mathbb{N}$ and $\pi \in S_n$ such that $(\bm{x}_{i})_\pi = a$. 
It follows that $x_{r,i}\neq 0$ for every $i\geq 1$ if and only if $P$ has no solution in $\mathbb{N}^n$. 
Thus, by Lemma \ref{lemma: equivalence iteration and strong Nullstellensatz} there exist an integer $r\geq 1$ and difference polynomials $f_1,\ldots,f_\ell,g\in \mathbb{Q}[\sigma^i(T_1),\ldots,\sigma^i(T_r) \mid i\in\mathbb{N}]\subseteq k_0[\sigma^i(T_1),\ldots,\sigma^i(T_r) \mid i\in\mathbb{N}]$ such that $g$ does not vanish on every solution of $f_1=\ldots=f_\ell=0$ in $k^M$ if and only if $P$ has no solution in $\mathbb{N}^n$.
\end{proof}

\begin{proof}[Proof of Corollary~\ref{cor:strong_undecidability}]
  The undecidability of~\ref{prob:1ineq} follows from Theorem~\ref{thm:stong_undecidable} and the fact that the system $f_1 = \ldots = f_\ell = 0,\ g\neq 0$ has a solution in $k^M$ if and only if $g = 0$ does not hold for some solution of $f_1 = \ldots = f_\ell = 0$ in $k^M$.
  
  Let $K$ be an uncoutable algebraically closed field containing $k$.
  Theorem~\ref{thm:nullstellensatz_over_constants} implies that 
  \[
        g \in \sqrt{ \langle \sigma^m(f_1), \ldots, \sigma^m(f_\ell) \mid m \in M \rangle }
  \]
  if and only if $g = 0$ vanishes on every solution of $f_1 = \ldots = f_\ell = 0$ in $K^M$.
  Thus, the undecidability of~\ref{prob:radical_ideal_membership} follows from Theorem~\ref{thm:stong_undecidable}.
\end{proof}


\subsection{Proof of Proposition~\ref{prop:2d}}

We will first consider the case $M = \mathbb{Z}^2$ and then reduce the case $M = \mathbb{N}^2$ to it.

Consider a set $\mathcal{D} = \{D_1, \ldots, D_n\}$ of dominos (in the sense of~\cite[p.~1]{Berger}) such that the labels on the edges are integers from $1$ to $N$.
We will construct a finite set $F \subset \mathbb{Q}[\sigma^m(X), \sigma^m(Y) \mid m \in \mathbb{Z}^2]$ such that the tilings of the plane by $\mathcal{D}$ correspond bijectively to the solutions of $F = 0$ in $k^{\mathbb{Z}^2}$.

For every $1 \leqslant i \leqslant n$, by $D_i(l), D_i(r), D_i(t)$, and $D_i(b)$ we denote the marks on the left, right, top, and bottom edges of $D_i$, respectively.
Let
\begin{multline}\label{eq:Z2system}
F := \{ (X - 1)(X - 2)\ldots(X - N),\;\; (Y - 1)(Y - 2)\ldots(Y - N),\\ \prod\limits_{k = 1}^n  \bigl((D_k(b) - X)^2 + (D_k(t) - \sigma^{(0, 1)}(X))^2 + (D_k(l) - Y)^2 + (D_k(r) - \sigma^{(1, 0)}(Y))^2\bigr)\}.
\end{multline}

Consider any tiling of the plane by dominos from $\mathcal{D}$. 
For every $i, j \in \mathbb{Z}$, we denote 
\begin{itemize}
    \item  the mark on the edge connecting the points $(i, j)$ and $(i + 1, j)$ by $x_{i, j}$;
    \item the mark on the edge connecting the points $(i, j)$ and $(i, j + 1)$ by $y_{i, j}$.
\end{itemize}
Then $((x_{i, j})_{i, j\in \mathbb{Z}}, (y_{i, j})_{i, j\in \mathbb{Z}})$ is a solution of $F = 0$ in $k^{\mathbb{Z}^2}$ because
\begin{itemize}
    \item all marks are integers from $1$ to $N$, so the first two polynomials in $F$ vanish
    \item and the last polynomial in $F$ vanishes if and only if each square is covered by a domino from $\mathcal{D}$.
\end{itemize}

For the other direction, let $((x_{i, j})_{i, j\in \mathbb{Z}}, (y_{i, j})_{i, j\in \mathbb{Z}})$ be a solution of $F = 0$ in $k^{\mathbb{Z}^2}$.
Then all $x_{i, j}$'s and $y_{i, j}$'s are integers from $1$ to $N$, so they are valid edge marks.
Moreover, if we mark the edges of the integer lattice by numbers $x_{i, j}$ and $y_{i, j}$ as described above, then the fact that $((x_{i, j})_{i, j\in \mathbb{Z}}, (y_{i, j})_{i, j\in \mathbb{Z}})$ satisfies the last equation in $F = 0$ implies that these marks produce a tiling by dominoes from $\mathcal{D}$.

Since the problem of determining whether there is a tiling of the plane by a given set of dominoes is undecidable~\cite[page 2]{Berger}, the problem of determining consistency of a system of $\mathbb{Z}^2$-poynomials in $k^{\mathbb{Z}^2}$ is also undecidable.

The undecidability of the consistency problem for $M = \mathbb{N}^2$ follows from the above argument and the following lemma.

\begin{lemma}
  Consider $F \subset \mathbb{Q}[ \sigma^m(X), \sigma^m(Y) \mid m\in \mathbb{N}^2]$ defined by~\eqref{eq:Z2system}.
  Then $F = 0$ has a solution in $k^{\mathbb{Z}^2}$ if and only if it has a solution in $k^{\mathbb{N}^2}$.
\end{lemma}

\begin{proof}
  Consider a solution of $F = 0$ in $k^{\mathbb{Z}^2}$. 
  If we restrict it on $\mathbb{N}^2$, we will obtain a solution of $F = 0$ in $k^{\mathbb{N}^2}$.
  
  Assume that $F = 0$ does not have a solution in $k^{\mathbb{Z}^2}$. 
  Let $K$ be an uncountable algebraically closed field containing $k$.
  The first two equations of $F = 0$ force all the coordinates of any solution of $F = 0$ in $K$ be integers from $1$ to $N$.
  Thus, $F = 0$ does not have a solution in $K^{\mathbb{Z}^2}$ as well.
  Then Theorem~\ref{thm:nullstellensatz_over_constants} implies that $1$ belongs to the $\mathbb{Z}^2$-invariant ideal generated by $F = \{f_1, f_2, f_3\}$, that is, there exists a positive integer $H$ such that 
  \begin{equation}\label{eq:expansion1}
  1 = \sum\limits_{\ell = 1}^3 \left( \sum\limits_{-H \leqslant i, j \leqslant H} c_{i, j} \sigma^{(i, j)}(f_\ell) \right),
  \end{equation}
  where $c_{i, j} \in K[\sigma^m(X), \sigma^m(Y) \mid m\in \mathbb{Z}^2]$ and $-H \leqslant a, b \leqslant H$ for every $\sigma^{(a, b)}$ appearing in $c_{i, j}$.
  Acting by $\sigma^{(H, H)}$ on~\eqref{eq:expansion1}, we conclude that $1$ belongs to the $\mathbb{N}^2$-invariant ideal generated by $F$ in $K[\sigma^m(X), \sigma^m(Y) \mid m\in \mathbb{N}^2]$.
  Thus, $F = 0$ does not have solutions in $k^{\mathbb{N}^2}$.
\end{proof}


\subsection{Proof of Proposition~\ref{prop:free_monoid}}

We will prove Proposition~\ref{prop:free_monoid} by reducing to Corollary~\ref{cor:strong_undecidability}.
More precisely, for every set of difference polynomials $f_1, \ldots, f_\ell, g \in k_0[\sigma^i(\mathbf{X}) \mid i \in \mathbb{N}]$ with $\mathbf{X} = (X_1, \ldots, X_n)$, we will construct a system $F = 0$ of $M_2$-polynomials over $k_0$ such that there exists a solution of $f_1 = \ldots = f_\ell = 0,\; g\neq 0$ in $k^{\mathbb{N}}$ if and only if $F = 0$ has a solution in $k^{M_2}$.

By adding new variables and equations, we may assume that $g \in k_0[\mathbf{X}]$  and $f_1, \ldots, f_\ell \in k_0[\mathbf{X}, \sigma(\mathbf{X})]$.
Let $\mathbf{Y} = (Y_1, \ldots, Y_{n})$, and denote the generators of $M_2$ by $a$ and $b$.
From $f_1, \ldots, f_\ell, g$, we obtain $\tilde{f}_1, \ldots, \tilde{f}_\ell, \tilde{g} \in k_0[\sigma^m(\mathbf{Y}), \sigma^m(Z) \mid m \in M_2]$ by replacing every $\sigma$ by $\sigma^a$ and every $X_i$ by $Y_i$.
Then we set
\[
  F := \{\tilde{f}_1, \ldots, \tilde{f}_\ell, Z\sigma^b(\tilde{g}) - 1\}.
\]
Let $(\bm{y}_m, z_m)_{m \in M_2}$ be a solution of $F = 0$ in $k^{M_2}$.
Then $\tilde{f}_1 = \ldots = \tilde{f}_\ell = 0$ implies that 
$\{ \mathbf{y}_{ba^i}\}_{i \in \mathbb{N}}$ 
is a solution of $f_1 = \ldots = f_\ell = 0$ in $k^{\mathbb{N}}$.
Furthermore, the equation $Z\sigma^b(\tilde{g}) - 1 = 0$ implies that $g(\bm{y}_b) \neq 0$, so $g$ does not vanish on this solution.

Conversely, let $(\mathbf{x}_i)_{i \in \mathbb{N}}$ be a solution of $f_1 = \ldots = f_\ell = 0,\; g\neq 0$.
By applying $\sigma$ to it, we may further assume that $c := g(\mathbf{x}_0) \neq 0$.
For every $m \in M_2$, we denote with $A(m)$ the largest $i \in \mathbb{N}$ such that $m$ can be written as $m'a^i$ for some $m' \in M_2$.
For every $m \in M_2$, we define $\mathbf{y}_m := \mathbf{x}_{A(m)}$ and $z_m := c^{-1}$.
We claim that $(\mathbf{y}_m, z_m)_{m \in M_2}$ is a solution of $F = 0$.
Let $\mathbf{t}_0$ and $\mathbf{t}_1$ be $n$-tuples of new algebraic indeterminates.
For every $1 \leqslant i \leqslant \ell$, let $P_i \in k_0[\mathbf{t}_0, \mathbf{t}_1]$ be a polynomial such that $f_i(\mathbf{X}) = P_i(\mathbf{X}, \sigma(\mathbf{X}))$.
Then $\tilde{f}_i(\mathbf{Y}) = P_i(\mathbf{Y}, \sigma^a(\mathbf{Y}))$.
For every $m_0 \in M_2$, we have
\[
\tilde{f}_i((\mathbf{y}_m)_{m \in M_2})_{m_0} = P_i(\mathbf{y}_{m_0}, \mathbf{y}_{m_0a}) = P_i(\mathbf{x}_{A(m_0)}, \mathbf{x}_{A(m_0a)}) = P_i(\mathbf{x}_{A(m_0)}, \mathbf{x}_{A(m_0) + 1}) = f_i((\mathbf{x}_i)_{i \in \mathbb{N}})_{A(m_0)} = 0.
\]
Let $Q \in k_0[\mathbf{t}_0]$ be a polynomial such that $g(\mathbf{X}) = Q(\mathbf{X})$ and $\tilde{g}(\mathbf{Y}) = Q(\mathbf{Y})$.
Then, for every $m_0 \in M_2$, we also have 
\[
\sigma^b(\tilde{g}((\mathbf{y}_m)_{m \in M_2}))_{m_0} = Q(\mathbf{y}_{m_0b}) = Q(\mathbf{x}_0) = g(\mathbf{x}_0) = c. 
\]
This proves the claim.


\subsection{Proof of Theorem~\ref{theo:counter strong_in_main}}
\label{subsec:Counterexample}

In this section we present an example that shows that the assumption $|k|>|M|$ cannot be omitted from Theorem \ref{thm:nullstellensatz_over_constants}. In more detail, we present a finite system $F\subseteq \overline{\mathbb{Q}}[\sigma^i(T_1),\ldots,\sigma^i(T_r) | \ i \in \mathbb{N} ]$ of difference polynomials (with respect to $M=\mathbb{N}$) such that $\mathcal{I}(\mathcal{V}(F))\supsetneqq\sqrt{\langle \sigma^i(F) \mid i \in \mathbb{N}\rangle}$.

Before going into the details of the construction of $F$ we explain the underlying ideas. 
Very roughly, the idea is to construct a piecewise polynomial map $\mathbf{p} \colon \mathbb{A}^n_{\mathbb{Q}}\to\mathbb{A}^n_{\mathbb{Q}}$ that can detect if a given number is algebraic or transcendental and then to obtain $F$ from $\mathbf{p}$ via Lemma \ref{lemma: equivalence iteration and strong Nullstellensatz}.
More precisely, we will proceed in the following steps:
\begin{enumerate}[label=(\alph*)]
    \item Construct a piecewise polynomial map $\mathbf{p} \colon \mathbb{A}^n_{\mathbb{Q}}\to\mathbb{A}^n_{\mathbb{Q}}$
such that for $\mathbf{x}_0=(c,0,\ldots,0,1)\in \mathbb{C}^n$ and $\mathbf{x}_{i+1}=\mathbf{p}(x_i)$ we have the following property: 
\[
  \text{the sequence } (x_{n,i})_{i\in\mathbb{N}} \text{ contains } 0 \iff c \in \overline{\mathbb{Q}}.
\]
    \item Apply Lemma~\ref{lemma: equivalence iteration and strong Nullstellensatz} with $V = \mathbb{A}^1_{\mathbb{Q}}\times\{0\}\times\ldots\times\{0\}\times \{1\} \subseteq \mathbb{A}^n_{\mathbb{Q}}$ and $\mathbf{p}$ being the map constructed in the previous step.
    This gives rise to difference polynomials $f_1,\ldots,f_\ell,g\in\mathbb{Q}[\sigma^i(T_1),\ldots,\sigma^i(T_r)|\ i\in\mathbb{N}]$ such that for every field extension $K$ of $\mathbb{Q}$ the following are equivalent: 
    \begin{itemize}
        \item $g$ vanishes on every solution of $f_1=\ldots=f_\ell=0$ in $(K^\mathbb{N})^r$;
        \item $K \subseteq \overline{\mathbb{Q}}$.
    \end{itemize}
    \item Taking $K=\overline{\mathbb{Q}}$, we see that $g\in \mathcal{I}(\mathcal{V}(F))$. 
    On the other hand, since there is a solution of $f_1=\ldots=f_\ell=0$ in $(\mathbb{C}^\mathbb{N})^r$, on which $g$ does not vanish, we conclude that $g\notin\sqrt{\langle \sigma^i(F) \mid i \in \mathbb{N}\rangle}$. 
\end{enumerate}

The piecewise polynomial map $\mathbf{p} \colon \mathbb{A}^n_{\mathbb{Q}}\to\mathbb{A}^n_{\mathbb{Q}}$ is explicitly given below (indeed we will see that one can choose $n=5$) and the proof of Lemma \ref{lemma: equivalence iteration and strong Nullstellensatz} is constructive. So, in principle it would be possible to explicitly determine $r$, $F=\{f_1,\ldots,f_\ell\}\subseteq \overline{\mathbb{Q}}[\sigma^i(T_1),\ldots,\sigma^i(T_r)|\ i\in\mathbb{N}]$ and $g\in \mathcal{I}(\mathcal{V}(F))\smallsetminus \sqrt{\langle \sigma^i(F) \mid i \in \mathbb{N}\rangle}$. However, since the piecewise polynomial map $\mathbf{p} \colon \mathbb{A}^5_{\mathbb{Q}}\to\mathbb{A}^5_{\mathbb{Q}}$ is already fairly complicated, this would be a very tedious task, yielding an enormously large system $F$. Moreover, we do not expect any deeper insight from determining $F$ explicitly.

We will next define the piecewise polynomial map $\mathbf{p} \colon \mathbb{A}^5_{\mathbb{Q}}\to\mathbb{A}^5_{\mathbb{Q}}$ 
that detects whether or not a given number is algebraic. 
Again, we first explain the underlying idea. 
The piecewise polynomial map $\mathbf{p} \colon \mathbb{A}^5_{\mathbb{Q}}\to\mathbb{A}^5_{\mathbb{Q}}$
should have the following property: 
If $K$ is a field extension of $\mathbb{Q}$, $\mathbf{x}_0=(c,0,0,0,1)\in K^5$ and $\mathbf{x}_{i+1}=\mathbf{p}(x_i)$, then $(x_{5,i})_{i\in\mathbb{N}}$ contains $0$ if and only if $c$ is algebraic. 
This property will be satisfied if the sequence $x_{5,i}$ consists of all expressions of 
the form $P(c)$, where $P$ ranges over all nonzero polynomials in $\mathbb{Z}[x]$. 
To achieve the latter, we will generate all elements of $\mathbb{Z}[x]$ under iteration. 
We use the observation that, up to multiplication with $\pm 1$, every element of $\mathbb{Z}[x]$ can be obtained from $1$ by iterating the following three operations (in a specific order): 
$P\mapsto P+1$, $P\mapsto xP$, $P\mapsto -xP$. 
We formulate a more precise statement in Lemma~\ref{eqn: def of Pa} below.

We set $P_{\varnothing}(x) = 1$ and for $a=(a_m,\ldots,a_0)\in \{0,1,2\}^{m+1}$ we define
$P_a(x)\in \mathbb{Z}[x]$ recursively by
\begin{equation} \label{eqn: def of Pa}
P_a(x)=\begin{cases}
xP_{a'}(x) \text{ if } a_m=0, \\
-xP_{a'}(x) \text{ if } a_m=1, \\
P_{a'}(x)+1 \text{ if } a_m=2,
\end{cases}
\end{equation}
where $a'=(a_{m-1},\ldots,a_0)$ (if $m = 0$, $a' = \varnothing$). 
For $N\in\mathbb{N}$ with base $3$ expansion $N=a_m3^m+a_{m-1}3^{m-1}+\ldots+a_0$, i.e., $a_0,\ldots,a_m\in\{0,1,2\}$ and $a_m\neq 0$, we set $P_N(x)=P_a(x)$ for $a=(a_m,\ldots,a_0)$.
For $N = 0$, we set $P_N(x) = P_\varnothing(x) = 1$.

\begin{lemma}\label{lem:possibe_P_N}
  For every nonzero polynomial $q(x) \in \mathbb{Z}[x]$, there exists an integer $N \geqslant 0$ such that $P_N(x)$ is equal to $q(x)$ or $-q(x)$.
\end{lemma}

\begin{proof}
The set of polynomials in $\mathbb{Z}[x]$ that can be obtained from $1$ by a finite sequence of the three operations $P(x) \mapsto xP(x)$, $P(x) \mapsto -xP(x)$, and $P(x) \mapsto P(x) + 1$ is the set of nonzero polynomials in $\mathbb{Z}[x]$ whose constant coefficient is nonnegative. Thus, up to multiplication with $\pm 1$ every nonzero polynomial in $\mathbb{Z}[x]$ can be obtained in this way.

The set of all $P_N(x)$'s consists of all polynomials in $\mathbb{Z}[x]$ that can be obtained from $1$ by a finite sequence of the three operations $P(x) \mapsto xP(x)$, $P(x) \mapsto -xP(x)$, and $P(x) \mapsto P(x) + 1$ under the additional assumption that the last operation is not $x\mapsto xP(x)$. This extra condition comes from the fact that in the base $3$ expansion $N=a_m3^m+a_{m-1}3^{m-1}+\ldots+a_0$ of $N$ one necessarily has $a_m\neq0$.

Let $q(x)\in\mathbb{Z}[x]$ be a nonzero polynomial. Multiplying $q(x)$ with $-1$ if necessary, we may assume that the constant coefficient of $q(x)$ is nonnegative. Thus, as observed above, $q(x)=P_a(x)$ for a suitable tuple $a=(a_m,\ldots,a_0)\in\{0,1,2\}^{m+1}$. If $a_m\neq 0$, then $q(x)=P_a(x)=P_N(x)$ for $N=a_m3^m+a_{m-1}3^{m-1}+\ldots+a_0$. If $a_m=0$, then $q(x)=-P_{\widetilde{a}}(x)=-P_{\widetilde{N}}(x)$ for $\widetilde{a}=(1,a_{m-1},\ldots,a_0)$ and $\widetilde{N}=1\cdot 3^m+a_{m-1}3^{m-1}+\ldots+a_0$.
\end{proof}

Now that we know how to iteratively produce all nonzero polynomials of $\mathbb{Z}[x]$, at least up to sign, we return to the definition of the piecewise polynomial map $\mathbf{p} \colon \mathbb{A}^5_{\mathbb{Q}}\to\mathbb{A}^5_{\mathbb{Q}}$ that should detect whether or not a given number $c$ is algebraic. 
The idea to produce all the $P_N(c)$'s as the entries of the sequence $x_{5,i}$ is to have one coordinate, say 
the second coordinate, 
loop through all the natural numbers $N$, while two other coordinates, 
say 
the third and the fourth coordinate, 
are used to compute the base $3$ expansion of $N$. 
This base $3$ expansion is then used to create $P_N(c)$ in the fifth coordinate according to the rule from (\ref{eqn: def of Pa}).

The computation of the base $3$ expansion of a given natural number $N$ in the second coordinate works as follows: 
The fourth coordinate 
starts looping from $0$, with increments of $1$, until it reaches a natural number $A_1$ with the property that $N - 3A_1 \in \{0,1,2\}$. 
In other words, $N-3A_1=a_0$, where $N=a_m3^m+a_{m-1}3^{m-1}+\ldots+a_0$ is the base $3$ expansion of $N$. 
Then $A_1$ is stored in 
the third coordinate 
and 
the fourth coordinate 
starts looping again from $0$, with increments of $1$, until it reaches a natural number $A_2$ with the property that $A_1 - 3A_2\in\{0,1,2\}$, i.e., $A_1 - 3A_2 = a_1$. 
Then $A_2$ is stored in the third coordinate and the process continues like this until we reach the index $m$, such that $A_m\in \{0,1,2\}$, i.e., $A_m = a_m$. 
At this point the full base $3$ expansion of $N$ has been computed and we start over with $N$ replaced by $N+1$.

Explicitly, the piecewice polynomial map $\mathbf{p} \colon \mathbb{A}^5_{\mathbb{Q}}\to\mathbb{A}^5_{\mathbb{Q}}$ is defined as $\mathbf{p}=(C,N,R,A,P)$, where 
$Q(x) := x(x - 1)(x - 2)$ and
\begin{align}
\begin{split}\label{eq:counterexample}
	C(\mathbf{x}) &= x_1,\\
	N(\mathbf{x}) &= \begin{cases}
		x_2 + 1, \text{ if } x_3 = 0,\\
		x_2, \text{ if } x_3\neq 0,
	\end{cases}\\
	R(\mathbf{x}) &= \begin{cases}
		x_2 + 1, \text{ if } x_3 = 0,\\
		x_3, \text{ if } x_3\neq 0 \ \& \ Q(x_3 - 3x_4) \neq 0,\\
			x_4, \text{ if } x_3\neq 0 \ \& \ Q(x_3 - 3x_4)=0 ,
	\end{cases}\\
	A(\mathbf{x}) &= \begin{cases}
	0, \text{ if } x_3=0, \\
		x_4 + 1, \text{ if }  x_3\neq 0 \ \& \ Q(x_3 - 3x_4) \neq 0, \\
			0, \text{ if } x_3\neq 0 \ \& \ Q(x_3 - 3x_4) = 0,\\
	\end{cases}\\
	P(\mathbf{x}) &= \begin{cases}
		1, \text{ if } x_3 = 0,\\
		x_5, \text{ if }  x_3\neq 0 \ \& \ Q(x_3 - 3x_4) \neq 0,\\
			x_5 x_1, \text{ if } x_3 \neq 0 \;\&\; x_3 - 3x_4 = 0,\\
		-x_5 x_1, \text{ if } x_3\neq 0 \ \& \ x_3 - 3x_4 = 1,\\
		x_5 + 1, \text{ if } x_3\neq 0 \ \& \ x_3 - 3x_4 = 2.
	\end{cases}
\end{split}
\end{align}

\begin{lemma} \label{lemma: describe x5}
	 Let $K$ be a field of characteristic zero and $c\in K$. Set $\mathbf{x}_0=(c,0,0,0,1)$ and $\mathbf{x}_{i+1}=\mathbf{p}(\mathbf{x}_i)$ for $i\geq 0$.
	Then every entry of the sequence $(x_{5,i})_{i\in\mathbb{N}}$ is either equal to $1$ or equal to $P_a(c)$ for some $a=(a_m,\ldots,a_0)\in \{0,1,2\}^{m+1}$. Moreover, for $N\geq 1$, every $P_N(c)$ eventually occurs in the sequence $(x_{5,i})_{i\in\mathbb{N}}$.
\end{lemma}

\begin{proof}
The sequence $(x_{1,i})_{i\in\mathbb{N}}$ is constant with value $c$.
 The entries of the sequence $(x_{2,i})_{i\in\mathbb{N}}$ are in $\mathbb{N}$ and in the step $i\rightsquigarrow i+1$ the sequence remains constant or increases by one. We shall see that $(x_{2,i})_{i\in\mathbb{N}}$ eventually assumes every $N\in\mathbb{N}$. The sequences $(x_{3,i})_{i\in\mathbb{N}}$ and $(x_{4,i})_{i\in\mathbb{N}}$ also only take values in $\mathbb{N}$.
	
	Note that if $x_{3,i}\neq 0$ and $Q(x_{3,i}-3x_{4,i})\neq 0$, then in the step $i\rightsquigarrow i+1$ the value for $x_4$ increases by $1$ but the values of all the other $x_i$'s remain constant. Let us analyze what happens in the steps $i\rightsquigarrow i+1\rightsquigarrow i+2\ldots$ when $x_{3,i}=0$. Then the value for $x_2$ increases by $1$, say $x_{2,i+1}=N\geq 1$. We have $$\mathbf{x}_{i+1}=(c, N, N, 0, 1), \ \mathbf{x}_{i+2}=(c, N, N, 1, 1),\ \mathbf{x}_{i+3}=(c, N, N, 2, 1),\ldots$$ and this continues until we reach an $\ell_1\geq 1$ such that $a_0=N-3 x_{4,\ell_1}\in\{0,1,2\}$, i.e., until $x_{4,\ell_1}=\lfloor\frac{N}{3}\rfloor$. 
	Note that $a_0=N-3 x_{4,\ell_1}$ is the last coefficient in the base $3$ expansion $N=a_m3^m+\ldots+a_13+a_0$ of $N$. So $\mathbf{x}_{\ell_1}=(c,N,N,\lfloor\frac{N}{3}\rfloor,1)$ and because $x_{3,\ell_1}-3x_{4,\ell_1}=a_0\in\{0,1,2\}$ we have $x_{3,\ell_1}\neq 0$ and $Q(x_{3,\ell_1}-3x_{4,\ell_1})=0$. Thus, according to the definition of $\mathbf{p}$:
	$$\mathbf{x}_{\ell_1+1}=(c, N, \lfloor\tfrac{N}{3}\rfloor ,0,P_{a_0}(c)),\ \mathbf{x}_{\ell_1+2}=(c, N, \lfloor\tfrac{N}{3}\rfloor ,1,P_{a_0}(c)),\ \mathbf{x}_{\ell_1+3}=(c, N, \lfloor\tfrac{N}{3}\rfloor ,2,P_{a_0}(c)),\ldots$$
	and this continues until we reach an $\ell_2\geq \ell_1$ such that $a_1=\lfloor\tfrac{N}{3}\rfloor-3 x_{4,\ell_2}\in \{0,1,2\}$, i.e., until $x_{4,\ell_2}=\lfloor\tfrac{\lfloor\frac{N}{3}\rfloor}{3}\rfloor$. 
	So $\mathbf{x}_{\ell_2}=(c,N,\lfloor\frac{N}{3}\rfloor,\lfloor\tfrac{\lfloor\frac{N}{3}\rfloor}{3}\rfloor,P_{a_0}(c))$ and because $x_{3,\ell_2}-3x_{4,\ell_2}=a_1\in\{0,1,2\}$ we have
	 $$\mathbf{x}_{\ell_2+1}=(c, N, \lfloor\tfrac{\lfloor\frac{N}{3}\rfloor}{3}\rfloor,0,P_{(a_1,a_0)}(c)),\ \mathbf{x}_{\ell_2+2}=(c, N, \lfloor\tfrac{\lfloor\frac{N}{3}\rfloor}{3}\rfloor ,1,P_{(a_1,a_0)}(c)), \ldots$$
	 and so on, until we eventually reach an $\ell_m$ with $\ell_m\geq \ell_{m-1}\geq\ldots\geq \ell_1$, $a_{m-1}=x_{3,\ell_m}-3x_{4,\ell_m}\in\{0,1,2\}$ and $a_m=x_{4,\ell_m}\in\{1,2\}$. (The case $x_{4,\ell_m}=0$ does not occur because it contradicts the minimality of $\ell_m$.)
Then 
\[
\mathbf{x}_{\ell_m}=(c, N, a_m3+a_{m-1}, a_m, P_{(a_{m-2},\ldots,a_0)}(c))
\]
and because $x_{3,\ell_m}-3x_{4,\ell_m}=a_{m-1}\in\{0,1,2\}$ we have 
\[
\mathbf{x}_{\ell_m+1}=(c, N, a_m, 0, P_{(a_{m-1},\ldots,a_0)}(c)). 
\]
Since $x_{3,\ell_m+1}-3x_{4,\ell_m+1}=a_m\in\{1,2\}$ it follows from the definition of $\mathbf{p}$ that  $$ \ \mathbf{x}_{\ell_m+2}=(c, N, 0, 0, P_{(a_{m},\ldots,a_0)}(c))$$
and $$\mathbf{x}_{\ell_m+3}=(c, N+1, N+1, 0,1).$$
Thus the whole process repeats with $N$ incremented by $1$. Since $N=a_m3^m+\ldots+a_0$ the claim follows.
\end{proof}

Lemmas~\ref{lemma: describe x5} and~\ref{lem:possibe_P_N} imply the following corollary.

\begin{corollary} \label{cor: zero iff algebraic}
  	With notation as in Lemma \ref{lemma: describe x5} we have: The sequence $(x_{5,i})_{i\in\mathbb{N}}$ contains zero if and only if $c$ is algebraic over $\mathbb{Q}$. \qed
\end{corollary}

We are now prepared to prove Theorem~\ref{theo:counter strong_in_main}.

\begin{proof}[Proof of Theorem~\ref{theo:counter strong_in_main}]
As above, we consider the piecewise polynomial map $\mathbf{p} \colon \mathbb{A}^5_{\mathbb{Q}}\to\mathbb{A}^5_{\mathbb{Q}}$	given by $\mathbf{p} = (C, N, R, A, P)$ with $C,N,R,A,P$ defined in~\eqref{eq:counterexample}. 
Let $V$ denote the closed subset of $\mathbb{A}^5_{\mathbb{Q}}$ defined by $X_2=X_3=X_4=0, X_5=1$. 
According to Lemma \ref{lemma: equivalence iteration and strong Nullstellensatz}, there exists an integer $r\geq 1$, a finite system $F=\{f_1,\ldots,f_\ell\}\subseteq \mathbb{Q}[\sigma^i(T_1),\ldots,\sigma^i(T_r) \mid i\in\mathbb{N}]$, and a difference polynomial $g\in \mathbb{Q}[\sigma^i(T_1),\ldots,\sigma^i(T_r) \mid i\in\mathbb{N}]$ such that, for every field extension $K$ of $\mathbb{Q}$, the following two statements are equivalent: 
\begin{enumerate}[label = \textbf{(\roman*)}]
	\item There exists a sequence $(\mathbf{x}_i)_{i\in\mathbb{N}}=(x_{1, i}, \ldots, x_{5, i})_{i \in \mathbb{N}}\in (K^\mathbb{N})^5$ such that
	\[
	\mathbf{x}_0 \in V(K) , \quad \mathbf{x}_{i + 1} = \mathbf{p}(\mathbf{x}_i) \text{ for every } i \in \mathbb{N},
	\]
	and $x_{5,i}
	\neq 0$ for $i\geq 1$. 
		\item There exists a solution of $F = 0$ in $(K^\mathbb{N})^r$ such that $g$ does not vanish on this solution.
\end{enumerate}
Following Corollary \ref{cor: zero iff algebraic} we see that (i) does not hold for the field $K=\overline{\mathbb{Q}}$, whereas (i) does hold for the field $K=\mathbb{C}$ (or any transcendental extension of $\mathbb{Q}$). Thus, (for $K=\overline{\mathbb{Q}}$) we see that $g$ vanishes on every solution of $F=0$ in $(\overline{\mathbb{Q}}^\mathbb{N})^r$, i.e., $g\in \mathcal{I}(\mathcal{V}(F))$. Whereas (for $K=\mathbb{C}$) it follows that $g$ does not vanish on every solution of $F=0$ in $(\mathbb{C}^\mathbb{N})^r$. Since an element of $\sqrt{\langle \sigma^i(F) \mid i\in\mathbb{N}\rangle}$ vanishes on every solution of $F=0$ over any field extension of $\mathbb{Q}$, we deduce that $g\notin \sqrt{\langle \sigma^i(F) \mid i\in\mathbb{N}\rangle}$.
\end{proof}

\subsection*{Acknowledgements}
The authors would like to thank Olivier Bournez, Ivan Mitrofanov, Alexey Ovchinnikov, and Amaury Pouly for helpful discussions.   
The authors thank 
the anonymous referee for a 
close reading of an earlier 
version of this paper and for 
suggesting some improvements.
This work has been partially supported by NSF grants CCF-1564132, CCF-1563942, DMS-1760448, DMS-1760212, DMS-1760413, DMS-1853482, DMS-1853650 by
PSC-CUNY grants \#69827-0047, \#60098-0048 and by the Lise Meitner grant M 2582-N32 of the Austrian Science Fund FWF.


\bibliographystyle{abbrvnat}
\bibliography{bibdata}

\end{document}